\documentclass[smallextended]{svjour3}       
\smartqed  
\usepackage{graphicx}
\usepackage{latexsym,amsmath,amsfonts,amssymb,color,enumerate,diagbox,subfigure} 
\usepackage[ruled,linesnumbered]{algorithm2e}
\usepackage[colorlinks,linkcolor=blue,anchorcolor=blue,citecolor=blue,urlcolor=blue]{hyperref}
\usepackage{multirow,rotating,adjustbox}
\newcommand{\R}{\mathbb{R}}
\newcommand{\Z}{\mathbb{Z}}
\newcommand{\N}{\mathbb{N}}
\newcommand{\K}{\mathcal{K}}
\newcommand{\C}{\mathcal{C}}
\newcommand{\NC}{\mathcal{N}_{\mathcal{C}}}
\renewcommand{\S}{\mathcal{S}}
\newcommand{\V}{\mathcal{V}}
\newcommand{\defas}{\overset{\mathrm{def}}{=}}
\newcommand{\clgap}{\texttt{clgap}}
\newcommand{\gap}{\texttt{gap}}
\newcommand{\nLAP}{\texttt{nLAP}}

\newcommand{\UB}{\texttt{UB}}
\newcommand{\LB}{\texttt{LB}}
\newcommand{\dccuttypeI}{\texttt{dccut-type-I}}
\newcommand{\dccuttypeII}{\texttt{dccut-type-II}}
\newcommand{\sampleone}{\texttt{sample\_30\_0\_10}}
\newcommand{\sampletwo}{\texttt{sample\_10\_0\_10}}

\DeclareMathOperator{\argmin}{argmin}
\DeclareMathOperator{\argmax}{argmax}
\DeclareMathOperator{\dom}{dom}
\DeclareMathOperator{\epi}{epi}
\DeclareMathOperator{\co}{co}

\usepackage{tikz}
\tikzstyle{arrow} = [->,>=stealth]

\newcommand{\orcid}[1]{\href{https://orcid.org/#1}{\includegraphics[scale=1]{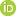}}}
\journalname{}
\begin{document}
\title{A Difference-of-Convex Cutting Plane Algorithm for Mixed-Binary Linear Program
	\thanks{The authors are supported by the National Natural Science Foundation of China (Grant 11601327) and by the Key Construction National ``$985$" Program of China (Grant WF220426001).}
}


\titlerunning{DCCUT Algorithm for MBLP}        

\author{Yi-Shuai Niu \orcid{0000-0002-9993-3681} \and Yu You}

\authorrunning{Yi-Shuai Niu \and Yu You} 

\institute{Yi-Shuai Niu \at
              School of Mathematical Sciences \& SJTU-Paristech, Shanghai Jiao Tong University, China \\
              \email{niuyishuai@sjtu.edu.cn}
           \and
           Yu You \at
           School of Mathematical Sciences, Shanghai Jiao Tong University, China \\
           \email{youyu0828@sjtu.edu.cn}
}

\date{Received: date / Accepted: date}

\maketitle

\begin{abstract}
In this paper, we propose a cutting plane algorithm based on DC (Difference-of-Convex) programming and DC cut for globally solving Mixed-Binary Linear Program (MBLP). We first use a classical DC programming formulation via the exact penalization to formulate MBLP as a DC program, which can be solved by DCA algorithm. Then, we focus on the construction of DC cuts, which serves either as a local cut (namely type-I DC cut) at feasible local minimizer of MBLP, or as a global cut (namely type-II DC cut) at infeasible local minimizer of MBLP if some particular assumptions are verified. Otherwise, the constructibility of DC cut is still unclear, and we propose to use classical global cuts (such as the Lift-and-Project cut) instead. Combining DC cut and classical global cuts, a cutting plane algorithm, namely DCCUT, is established for globally solving MBLP. The convergence theorem of DCCUT is proved. Restarting DCA in DCCUT helps to quickly update the upper bound solution and to introduce more DC cuts for lower bound improvement. A variant of DCCUT by introducing more classical global cuts in each iteration is proposed, and parallel versions of DCCUT and its variant are also designed which use the power of multiple processors for better performance. Numerical simulations of DCCUT type algorithms comparing with the classical cutting plane algorithm using Lift-and-Project cuts are reported. Tests on some specific samples and the MIPLIB 2017 benchmark dataset demonstrate the benefits of DC cut and good performance of DCCUT algorithms. 
\end{abstract}

\section{Introduction}\label{sec:intro}
\label{intro}
Mixed-Binary Linear Program, namely MBLP, is a well-known NP-hard problem raised as one of the Karp's $21$ NP-complete problems in \cite{Karp1972}. The general MBLP is given by:
\begin{equation*}
\begin{split}
\min \quad &f(x,y) \defas  c^{\top}x + d^{\top}y\\
\text{s.t.}\quad  &(x,y) \in \mathcal{K}\\
\quad & x \in \{0,1\}^n,
\end{split}
\label{MBLP}
\tag{MBLP}
\end{equation*}
where $n\in \N^*$ is the number of binary variables, $q\in \N$ is the number of continuous variables (i.e., the problem has at least one binary variable, and could have no continuous variable), the coefficients $c\in \R^{n}$ and $d\in \R^{q}$, the set $\K$ is supposed to be a nonempty compact polyhedral convex set of $\R^{n}\times \R^{q}$ in form of   
\begin{equation}
\label{eq:K}
\K\defas\{(x,y)\in [0,1]^n\times [0,\bar{y}] ~|~Ax+By \leq b\},
\end{equation} where $m\in \N$ is the number of linear constraints, the coefficients $A\in \R^{m\times n}$, $B\in \R^{m\times q}$, $b\in \R^m$, and $\bar{y} \in \R_{+}^q$. Note that any \eqref{MBLP} with negative variables in $y$ can be converted in form of positive variables only by classical standardization of linear program. The set of feasible solutions of \eqref{MBLP}, denoted by $\S$, is a subset of $\K$ and assumed to be nonempty.

The researches on MBLP can go back to the early 1950s with the birth of combinatorial integer programming in investigating the well-known traveling salesman problem (TSP) initialed by several pioneers: Hassler Whitney, George Dantzig, Julia Robinson, Ray Fulkerson, and Selmer Johnson \cite{Robinson1949,Dantzig1954,Dantzig1959}. There are very rich literatures on MBLP over the past seven decades, focusing either on theory and algorithms, or on real-world applications. Two classes of approaches exist for general MBLP: Exact methods and Heuristic methods. The exact methods consist of two typical approaches: continuous approaches and combinatorial approaches. The continuous approaches first reformulate the MBLP as nonconvex continuous optimization problems by replacing binary variables as continuous ones, then a classical nonlinear optimization algorithm such as Newton-type method \cite{Bertsekas1997}, gradient-type method \cite{Bertsekas1997}, difference-of-convex approache \cite{Le2001,Niu2008,Niu2010,Pham2016,Niu2018}, and dynamical system approach \cite{Niu2019DiscreteDS} is used for numerical solutions. The advantage of this type of methods is inexpensive local optimization procedure which helps to find quickly a computed solution with guaranteed local optimality at best. The combinatorial approaches exploit techniques such as cutting-plane \cite{Gomory1958,Balas1971,Balas1993,Cornuejols2008}, branch-and-bound/branch-and-cut/branch-and-price \cite{Dantzig1959,Efroymson1966,Kolesar1967}, and column generation \cite{Benders1962} etc., which aim at dividing big problems into smaller problems that are easier to solve with reduced problem size and shrinked search region. These methods can find exact global optimal solution for problems with moderate size. But due to the NP-hardness, they are often very computationally expensive for large-scale cases. To quickly find feasible solutions, heuristic methods are often preferred, such as genetic algorithms, simulated annealing, ant colony, tabu search, and neural networks etc. These approaches mimic activities in nature for searching potentially good candidates without guaranteeing to be optimal, which aim at providing a satisfactory computed solution (may not even be a local optimum) within a reasonable time, but the optimality of the computed solution is not qualified. These methods are often used to provide good feasible initial candidates, or for hard problems in which no exact method works to find any solution. All of these techniques have been widely used in many modern integer/mixed-integer optimization solvers such as GUROBI \cite{Gurobi}, CPLEX \cite{Cplex}, BARON \cite{Sahinidis1996}, MOSEK \cite{Mosek}, MATLAB intlinprog \cite{MATLAB}, and some open-source solvers as SCIP \cite{Achterberg2009}, BONMIN \cite{Bonmin} and COUENNE \cite{Belotti2009}. 

In our paper, we are interested in developing a hybrid approach combining techniques in continuous approach (DCA) and combinatorial approach (cutting planes) for globally solving MBLP. \textit{Concerning continuous approach}, we are interested in DCA which is a promising algorithm to provide good local optimal solution with many successful practical applications (see e.g., \cite{LeThi2005,LeThi2018,Pham1997,Pham1998}). Applying DCA to binary linear program is first introduced in \cite{LeThi2001} by H.A. Le Thi and D.T. Pham, then the general cases with mixed-integer linear/nonlinear optimization are studied by Y.S. Niu, D.T. Pham et al. (see e.g., \cite{Niu2008,Niu2010,Niu2011}) where the integer variables are not supposed to be binary. These methods are based on continuous representation techniques for integer set, exact penalty theorem, DCA and branch-and-bound. There are various applications of this kind of approaches including sentence compression \cite{niu2021sentence}, scheduling \cite{paper_lethi_2009a}, network optimization \cite{Schleich2012}, cryptography \cite{paper_lethi_2009c} and finance \cite{paper_lethi_2009b,Pham2016} etc. Recently, Y.S. Niu (the first author of the paper) released an open-source optimization solver, combining DCA with parallel branch-and-bound framework, namely PDCABB \cite{Niu2018} (code available on Github \cite{PDCABB}), for general mixed-integer nonlinear optimization. \textit{Concerning combinatorial approach}, we focus on DC cutting plane technique constructed at local minimizers provided by DCA. DC cut is proposed in \cite{Nguyen2006} with a DCA-CUT algorithm, and applied in several real-world problems such as the bin-parking problem \cite{Babacar2012} and the scheduling problem \cite{Quang2010}. However, DCA-CUT algorithm is not well-constructed in the sense that there exist some cases, often encountered in practice, where DC cut is unconstructible. Due to this drawback, there are several works (e.g., \cite{Niu2012,Quang2010}) combining constructible DC cuts, DCA and branch-and-bound for global optimal solutions of MBLP.

Our \textit{contributions} are: (1) Revisit DC cutting plane (cf. DC cut) techniques and provide more theoretical results on constructible and unconstructible cases for DC cuts. As a result, two types of DC cuts: type-I DC cut (local DC cut) and type-II DC cut (global DC cut) are established at local minimizers of DC programs; (2) Propose using classical global cuts, such as the Lift-and-Project (L\&P) cut, Gomory's mixed-integer (GMI) cut and mixed integer rounding (MIR) cut, when DC cut is unconstructible, to cut off infeasible point; (3) Establish a cutting plane algorithm, namely DCCUT algorithm, combining DC cut and classical global cuts for globally solving MBLP. The convergence theorem of DCCUT is also proved. (4) Implement DCCUT algorithm, its variant DCCUT-V1 with more cutting planes in each iteration, and parallel algorithms (P-DCCUT and P-DCCUT-V1) in a MATLAB toolbox and shared on Github \url{https://github.com/niuyishuai/DCCUT}. 

The rest of the paper is organized as follows: In Section \ref{sec:DCform}, we presents general results on DC programming formulation and DCA for MBLP. In Section \ref{sec:DCcut}, we investigate in detail the DC cut techniques, and briefly introduce the classical Lift-and-Project cut. The proposed DCCUT algorithm and its variant (with and without parallelism) are established in Section \ref{sec:DCCUTAlgo}. Numerical results are reported in Section \ref{sec:Experiments}. Concluding remarks and related questions are discussed in the last section.

\section{DC programming approach for MBLP}\label{sec:DCform}

In this section, we first introduce some backgrounds on Difference-of-Convex (DC) programming and DCA algorithm for readers to better understand the fundamental tool used in this paper, then we present the DC formulation for problem \eqref{MBLP} based on exact penalty theorem, and apply DCA to solve it.

\subsection{Notations}
Let $\R^l$ denote the $l$-dimensional Euclidean space equipped with the standard inner product $\langle \cdot, \cdot \rangle$ and the induced norm $\Vert \cdot \Vert$. 
For any extended real-valued function $\phi:\R^l\rightarrow(-\infty,+\infty]$, the effective domain is given by $\dom \phi \defas \{x\in\R^{l}~|~\phi(x) < +\infty \}$. If $\dom \phi \neq \emptyset$, then $\phi$ is called a proper function. The directional derivative of $\phi$ at $x\in \dom\phi$ along a direction $d\in \R^l$ is defined by
\begin{equation*}
\phi'(x;d) = \lim\limits_{\tau \downarrow 0} \frac{\phi(x+\tau d)-\phi(x)}{\tau}.
\end{equation*}
The epigraph of $\phi$ is defined by
$\epi \phi\defas\{(x,t)\in \R^l \times \R~|~\phi(x)\leq t\}$, and $\phi$ is called a closed (resp. convex) function if $\epi \phi$ is closed (resp. convex). In particular, $\phi$ is called a polyhedral function if $\epi \phi$ is a polyhedral convex set. Given a nonempty polyhedral convex set $\K \subseteq \R^l$, the vertex set of $\K$ is denoted by $V(\mathcal{K})$ whose cardinality $|V(\mathcal{K})|$ is clearly finite. The set of all proper, closed and convex functions on $\R^l$ is denoted by $\Gamma_0(\R^l)$, which is a convex cone; and $\text{DC}(\R^l)\defas\Gamma_0(\R^l) - \Gamma_0(\R^l)$ is the set of DC (Difference-of-Convex) functions (under the convention that $(+\infty) - (+\infty) = +\infty$), which is in fact a vector space spanned by $\Gamma_0(\R^l)$. Let $\phi$ be a proper convex function, the subdifferential of $\phi$ at $x$ is defined by:
$$\partial \phi(x) \defas \{u\in \R^l~|~\phi(y) \geq \phi(x)+\langle u,y-x \rangle \;\text{for all}\;y\in \R^l\},$$ where any point in $\partial \phi(x)$ is called a subgradient of $\phi$ at $x$.  Clearly, $\partial \phi(x) \neq \emptyset$ implies that
$x \in \dom \phi$; particularly, if $\phi$ is differentiable at $x\in \dom \phi$, then $\partial \phi(x)$ is reduced to a singleton, i.e., $\partial \phi (x) = \{\triangledown \phi(x)\}$. 

Let $\C \subseteq \R^l$ be a nonempty convex set, the indicator function of $\C$, denoted by $\chi_{\mathcal{C}}:\R^l\to \R$, equals to $0$ on $\C$ and $+\infty$ otherwise. For any $x\in \C$, the set $\NC(x)\defas\{d\in \R^l~|~\langle d,y-x\rangle \leq 0 \text{ for all } y \in \C\}$ denotes the normal cone of $\C$ at $x$, and $\NC(z)=\emptyset$ if $z\notin \C$ by convention.

\subsection{DC program and DCA}
The standard DC program is defined by  
\begin{equation}
\label{primal1}
\alpha = \text{min}\; \{\phi(x) \defas \phi_1(x) - \phi_2(x):x\in\R^l\}, \tag{$P_{dc}$}
\end{equation}
where $\phi_1$ and $\phi_2$ belong to $\Gamma_0(\R^l)$ and $\phi\in \text{DC}(\R^l)$. The objective value $\alpha$ is assumed to be finite, so that $\emptyset \neq \dom \phi_1 \subset \dom \phi_2$. A point $x^*$  is called a \textit{critical point} of problem (\ref{primal1}) if $\partial \phi_1(x^*)  \cap \partial \phi_2(x^*) \neq \emptyset$, and a \emph{strongly critical point} of problem (\ref{primal1}) if $\emptyset \neq \partial \phi_2(x^*) \subset \partial \phi_1(x^*)$, i.e., $\phi'(x^*;x-x^*) \geq 0$ for all $x\in \dom\phi_1$.  Clearly, a strongly critical point must be a critical point, but the converse is not true in general. Note that in nonconvex optimization, a critical point or a strongly critical point of problem \eqref{primal1} may not be a local minimizer. We have some particular cases in which a stationary point is guaranteed to be a local minimizer of $\phi$, e.g., if $\phi$ is locally convex at a critical point $x^*$, then $x^*$ must be a local minimizer of $\phi$. The readers can refer to \cite{Pang2016} for more details about the stationarity. 

A well-known DC algorithm, namely DCA, for \eqref{primal1} was first introduced by D.T. Pham in 1985 as an extension of the subgradient method, and extensively developed by H.A. Le Thi and D.T. Pham since $1994$ (see e.g., \cite{Pham1997,Pham1998,LeThi2005,LeThi2018}). Broadly speaking, DCA is not one single algorithm, it is a philosophy to solve DC programs by a sequence of convex ones. A classical DCA gives:
$$x^{k+1}\in \argmin \{ \phi_1(x) - \langle x, y^k \rangle ~|~x\in \R^l \}$$
with $y^k\in \partial \phi_2(x^k)$. This convex program is derived from minimizing a convex overestimation of $\phi$ at iterate $x^k$, denoted by $\phi^k$, constructed by linearizing $\phi_2$ at $x^k$ as 
$$
\forall x\in \R^l,~ \phi(x) = \phi_1(x)-\phi_2(x)\leq \phi_1(x) - (\phi_2(x^k)+\langle x-x^k, y^k \rangle)\overset{\mathrm{def}}{=} \phi^k(x),
$$ 
where $y^k\in \partial \phi_2(x^k)$. The principal of DCA can be viewed geometrically in \cite{Niu2010,Niu2014}. See below the classical DCA for problem \eqref{primal1}:
\renewcommand{\algorithmcfname}{DCA for problem \eqref{primal1} }
\renewcommand{\thealgocf}{}
\begin{algorithm}[ht]
    \label{algo:DCA0}
	\caption{}
	
	\textbf{Initialization:} Choose $x^0 \in \dom(\partial \phi_2)$.
	
	\textbf{Iterations:} For $k = 0,1,2,\dots$
	
	Compute $y^k \in \partial \phi_2(x^k)$;
	
	Find $x^{k+1} \in \argmin\{ \phi_1(x)-\langle x,y^k \rangle, x\in\R^l\}.$	
\end{algorithm}

\begin{remark}
	To make sure that DCA can be implemented, it is required that at iteration $k$ both $\partial \phi_2(x^k)$ and $\argmin\{ \phi_1(x)-\langle x,y^k \rangle, x\in\R^l\}$ are nonempty. 
\end{remark}
The next theorem summarizes some important convergence results of DCA.
\begin{theorem}[Convergence of DCA \cite{Pham1997,LeThi2005}]
	\label{thm:1}
	Let $\{x^k\}$ and $\{y^k\}$ be bounded sequences generated by DCA Algorithm \ref{algo:DCA0} for problem \eqref{primal1}, and suppose that $\alpha$ is finite, then
	\begin{enumerate}
		\item The sequence $\{\phi(x^k)\}$ is decreasing and bounded below.
		\item Any limit point of $\{x^k\}$ is a critical point of problem \eqref{primal1}.
		\item If either $\phi_1$ or $\phi_2$ is polyhedral, then problem $\eqref{primal1}$ is called polyhedral DC program, and DCA is terminated in finitely many iterations.
		\item If $x^*$ is a critical point generated by DCA, and if $\phi$ is locally convex at $x^*$, then $x^*$ is a local minimizer of $\phi$.
	\end{enumerate}
\end{theorem} 

More theoretical results of DC program and DCA can be found in \cite{Pham1997,Pham1998,LeThi2005,LeThi2018} and the references therein.

\subsection{DC formulation for \eqref{MBLP}}
We will show that problem \eqref{MBLP} can be equivalently represented as a standard DC program in form of \eqref{primal1}. 

Firstly, we will use continuous representation technique to reformulate the binary set $\{0,1\}^n$ as a set involving continuous variables and continuous functions only. A classical way is using a nonnegative concave function $p$ over $[0,1]^n$ to rewrite the binary set as:
$$\{0,1\}^n = \{x\in [0,1]^n~|~p(x)\leq 0\}.$$
Then $$\S = \{(x,y)\in \K~|~ p(x)\leq 0\},$$
where the polyhedral convex set $\K$ is defined in \eqref{eq:K}. 
There are many alternative functions for $p$, such as the quadratic function $p:x\mapsto \sum_{i=1}^{n}x_i(1-x_i)$ and the piecewise linear function $p:x\mapsto \sum_{i=1}^{n} \min\{x_i,1-x_i\}$. The main differences between these two functions are: the quadratic function is differentiable and concave over $[0,1]^n$; while the piecewise linear function is not differentiable at some points (e.g., the points with some coordinates as $0.5$) but it is \emph{locally convex} (specifically, locally affine) at any differentiable point over $[0,1]^n$. In this paper, we will use the latter function since the local convexity is crucial to identify a local minimizer for a critical point returned by DCA.

Next, we will use the well known exact penalty theorem \cite{penalty1999,LeThi2012} to reformulate problem \eqref{MBLP} as a standard DC program. The exact penalty theorem is stated as follows:
\begin{theorem}[See e.g., \cite{penalty1999,LeThi2012}]
	\label{thm:exactpenalty}
	Let $p(x)= \sum_{i=1}^{n} \min\{x_i,1-x_i\}$, then there exists a finite number $t_0 \geq 0$ such that for all $t > t_0$, problem (\ref{MBLP}) is equivalent to
	\begin{equation}
	\label{Pt}
	\alpha_{t} = \min \{f(x,y) + t p(x) ~|~ (x,y) \in \mathcal{K}\}.
	\tag{$P_t$}
	\end{equation}
\end{theorem}
 
The equivalence means that problems (\ref{MBLP})
and (\ref{Pt}) have the same set of global optimal solutions. The penalty parameter $t_0$ can be computed by: 
\begin{equation}
    \label{eq:t0}
t_0 = \frac{\min\{f(x,y)~|~(x,y)\in \S\} -\alpha_0}{m},
\end{equation}
where $m = \min \{p(x)~|~ (x,y)\in V(\mathcal{K}), p(x) > 0\}$ under the convention that $m=+\infty$ if $\{(x,y)\in V(\mathcal{K})~|~p(x) > 0\} = \emptyset$ and $\frac{1}{+\infty} = 0$. In practice, computing $t_0$ is difficult since both $\min\{f(x,y)~|~(x,y)\in \S\}$ and $\min \{p(x)~|~ (x,y)\in V(\mathcal{K}), p(x) > 0\}$ are nonconvex optimization problems which are difficult to be solved, while the computation of $\alpha_0$ is easy which just need to solve a linear program. Note that if the set of vertices $V(\K)$ is known, then it will be much easier to find an upper bound for $t_0$, since $m$ can be easily solved by checking all points in $V(\K)$, and we just need to find a feasible solution of $\min\{f(x,y)~|~(x,y)\in \S\}$ to get an upper bound for $t_0$. In numerical simulations, the parameter $t$ is often fixed arbitrarily to be a large positive number. 

Supposing that $t$ is large enough, then problem \eqref{Pt} is in fact a concave minimization problem over a polyhedral convex set which is equivalent to a standard DC program as:
$$\min\{ \tau_t(x,y) ~|~ (x,y)\in \R^n\times \R^q\},$$
where $$\tau_t(x,y) = \underbrace{\chi_{\mathcal{K}}(x,y)}_{g(x,y)} - [\underbrace{-f(x,y)-tp(x)}_{h(x,y)}]$$
is a DC function in $\text{DC}(\R^n\times \R^q)$, and $g$ and $h$ are $\Gamma_0(\R^n\times \R^q)$ functions.

\subsection{DCA for Problem \eqref{Pt}}
To apply DCA for problem \eqref{Pt}, we first compute $\partial h$ by

$$\partial h(x,y) = (-c + tz, -d), \text{ where } z_i = \begin{cases}
1 & \text{, if } x_i >\frac{1}{2};\\
s \in [-1,1] & \text{, if } x_i = \frac{1}{2};\\
-1 & \text{, if } x_i < \frac{1}{2};
\end{cases}$$
for all $i=1,\ldots,n$.
Then, we can compute $(x^{k+1},y^{k+1})$ from $(x^k,y^k)$ by solving the linear program via simplex method:
$$(x^{k+1},y^{k+1}) \in \argmin \{ -\langle(v^k,w^k), (x,y) \rangle:(x,y)\in \mathcal{K} \},$$
where $(v^k,w^k)\in \partial h(x^k,y^k)$. Note that only $v^k$ is updated in each iteration and $w_k$ is fixed to $-d$. 

DCA could be terminated if $|\tau_t(x^{k+1},y^{k+1}) - \tau_t(x^k,y^k)|/(|\tau_t(x^{k+1},y^{k+1})|+1)\leq \varepsilon_1$ and/or $\|(x^{k+1},y^{k+1}) - (x^k,y^k)\|/(\|(x^{k+1},y^{k+1})\|+1)\leq \varepsilon_2$ where $\varepsilon_1$ and $\varepsilon_2$ are given tolerances. 

The detailed DCA for problem (\ref{Pt}) is summarized as follows. 
\renewcommand{\algorithmcfname}{DCA for problem (\ref{Pt})}
\renewcommand{\thealgocf}{}
\begin{algorithm}[ht]
	\label{algo:DCA}
	\caption{}
	\KwIn{Initial point $(x^0,y^0)\in \R^n\times \R^q$, small tolerances $\varepsilon_1>0$ and $\varepsilon_2>0$.}
	\KwOut{Computed solution $(x^*,y^*)$ and computed optimal value $\tau^*$.} 
	
	\textbf{Initialization:} Set $k \leftarrow 0$.
	
	\textbf{Step 1:} Compute $(v^k,w^k) \in \partial h(x^k,y^k)$.
	
	\textbf{Step 2:} Solve the linear program via the simplex algorithm: $$(x^{k+1},y^{k+1}) \in \argmin \{ -\langle(v^k,w^k), (x,y) \rangle:(x,y)\in \mathcal{K} \};$$

	$\Delta \tau \leftarrow |\tau_t(x^{k+1},y^{k+1}) - \tau_t(x^k,y^k)|/(|\tau_t(x^{k+1},y^{k+1})|+1);$ 

	$\Delta X \leftarrow \|(x^{k+1},y^{k+1}) - (x^k,y^k)\|/(\|(x^{k+1},y^{k+1})\|+1);$ 
	
	\textbf{Step 3:} Stopping check: 	

	\SetKwIF{If}{ElseIf}{Else}{if}{then}{else if}{else}{end}
	\eIf{$\Delta \tau \leq \varepsilon_1$ or $\Delta X\leq \varepsilon_2$}
	{
		$(x^*,y^*) \leftarrow (x^{k+1},y^{k+1})$; $\tau^* \leftarrow \tau_t(x^{k+1},y^{k+1})$; \Return;
	}
	{
		$k \leftarrow k+1$; \textbf{Goto} Step 1;
	}
\end{algorithm}

\begin{remark}
	\begin{enumerate}
		\item In Step 1, the selection of $(v^k,w^k)$ in $\partial h(x^k,y^k)$ may not be unique. When $x_i^k=1/2$, we can choose $u_i$ randomly in $[-1,1]$. 
		\item We suggest using the well-known simplex algorithm for solving the linear program required in Step 2 to find a vertex solution. It is easy to see that there exists optimal solution of \eqref{MBLP} in $V(\co(\S))$ (set of vertices of convex hull of $\S$), therefore, only vertex solutions are of our interests. 
	\end{enumerate}
\end{remark}

The convergence theorem of DCA for problem \eqref{Pt} is summarized in Theorem \ref{thm:1}. Note that problem \eqref{Pt} is a polyhedral DC program, thus DCA will converge in finitely many iterations. Concerning the local optimality of the computed solution returned by DCA, we have the following Proposition:

\begin{proposition}
	\label{prop:1}
		Let $(x^*,y^*)$ be a computed solution of problem \eqref{Pt} returned by DCA, if $x_i^* \neq 1/2, \forall i =1,\cdots,n$, then $(x^*,y^*)$ is a local minimizer of \eqref{Pt}.
\end{proposition}
\begin{proof}
	The function $h$ is differentiable at $(x^*,y^*)$ since all entries of $x^*$ are different to $\frac{1}{2}$, then $\partial h(x^*,y^*)$ is reduced to a singleton $\{\nabla h(x^*,y^*)\}$. Based on Theorem \ref{thm:1}, $(x^*,y^*)$ is a critical point, then $$\partial g(x^*,y^*) \cap \partial h(x^*,y^*) =  \{\nabla h(x^*,y^*)\} \subset \partial g(x^*,y^*) = \mathcal{N}_{\mathcal{K}}(x^*,y^*),$$ we get  
	\begin{equation}
	\label{E:1}
	\langle \nabla (-h)(x^*,y^*),(x,y)-(x^*,y^*) \rangle \geq 0, \;\forall (x,y)\in \mathcal{K}.
	\end{equation} 
	By the local convexity (locally affine) of $-h$ at $(x^*,y^*)$, there exists a neighborhood $U$ of $(x^*,y^*)$ such that 
	\begin{equation}
	\label{E:2}
	(-h)(x,y) \geq (-h)(x^*,y^*) + \langle \nabla (-h)(x^*,y^*),(x,y)-(x^*,y^*)\rangle,\;\forall (x,y)\in U.
	\end{equation}
	Combining \eqref{E:1} and \eqref{E:2}, we obtain
	\begin{equation*}
	(-h)(x,y) \geq (-h)(x^*,y^*), \forall (x,y)\in U\cap \mathcal{K}.
	\end{equation*}
	Therefore, $(x^*,y^*)$ is a local minimizer of problem (\ref{Pt}). \qed
\end{proof}

Note that if there exists $x_i^*=1/2$, then the critical point $(x^*,y^*)$ obtained by DCA may not be a local minimizer. For example, let $f(x,y)=0, \forall (x,y)\in \K = [0,1/2]^n\times \{0_{\R^q}\}$, then at the initial point $(x^0,y^0)$ with $x^0 = (\frac{1}{2},\ldots,\frac{1}{2})$ and $y^0 = 0_{\R^q}$, if we take $z=(1,\ldots,1)$, then $(x^1,y^1)\in \argmin \{-t\sum_{i=1}^{n} x_i ~|~ (x,y)\in \K \}$ yields $(x^1,y^1) = (x^0,y^0)$, thus DCA terminates at $(x^0,y^0)$ which is a critical point of \eqref{Pt}, but a global maximizer (not a local minimzer) of \eqref{Pt}.

\section{DC cutting planes}\label{sec:DCcut}

In this section, we will carefully revisit the construction of DC cut and point out clearly that when DC cut is unconstructible. Then, for unconstructible case, we propose to use classical global cuts (such as the Lift-and-Project cut) in order to construct a theoretically provable DCCUT algorithm which will be discussed in next section.

\subsection{Valid inequalities}\label{sec:3.1}
For convenience, we will denote $u=(x,y)$. Let $u^* \in \K$ and $I= \{1,\cdots,n\}$, then two complement subsets of indices related to $u^*$ is defined by:
\[J_0(u^*) \defas \{j \in I:x_j^* \leq \frac{1}{2}\}; ~J_1(u^*) \defas I\setminus J_0(u^*).\]
An affine function defined at $u^*$ is given by $l_{u^*}: \R^{n} \rightarrow \R$ such that
$$l_{u^*}(x) = \sum\nolimits_{i \in J_0(u^*)} x_i + \sum\nolimits_{j \in J_1(u^*)} (1-x_j).$$
Note that we identify the notations $l_{u^*}(x)$ and $p(x)$ by $l_{u^*}(u)$ and $p(u)$. The relationships between $l_{u^*}$ and $p$ are given as follows:

\begin{lemma}\label{lemma:1}
	Let $u^* \in \mathcal{K}$ and $p(x)= \sum_{i=1}^{n} \min\{x_i,1-x_i\}$, then we have
	\begin{enumerate}[(i)]
		\item $l_{u^*}(u^*) = p(u^*)$.
		\item $l_{u^*}(u) \geq p(u) \geq 0$, $\forall u\in \mathcal{K}$.
		\item If $u^* \in \mathcal{S}$, then $l_{u^*}(u^*) = p(u^*) = 0.$ 
		\item If $u^*\in \S$, then $\forall u\in \S$ with $x\neq x^*$, we have $l_{u^*}(u)>0.$
	\end{enumerate}	
\end{lemma}
\begin{proof}
$(i)$ By the definition of $l_{u^*}$ and $p$, we get immediately that 
$$l_{u^*}(u^*) = \sum_{i \in J_0(u^*)} x_i^* + \sum_{j \in J_1(u^*)} (1-x_j^*)= \sum_{i\in I} \min \{ x_i^*, 1-x_i^* \}= p(u^*).$$
$(ii)$ By the definition of $l_{u^*}$, $p$ and $\K \subset [0,1]^n\times [0,\bar{y}]$, we have 
$$\forall u\in \K,~ l_{u^*}(u) = \sum_{i \in J_0(u^*)} x_i + \sum_{j \in J_1(u^*)} (1-x_j)\geq \sum_{i\in I} \min \{ x_i, 1-x_i \}= p(u)\geq 0.$$
$(iii)$ If $u^*\in \S$, then $x^*\in \{0,1\}^n$ which implies $p(u^*) = 0$. It follows from $(i)$ the required equation $$l_{u^*}(u^*) = p(u^*) = 0.$$
$(iv)$ If $u^*\in \S$, then $\forall u\in \S$ with $x\neq x^*$, let $I$ be the set of all indices such that $\forall i\in I, x_i\neq x_i^*$, so that $|I|>0$. Let $I=I_0\cup I_1$ with $I_0\subset J_0(u^*)$ and $I_1\subset J_1(u^*)$. Then
$$l_{u^*}(u)  = \underbrace{\sum_{i \in I_0} x_i}_{=|I_0|} + \underbrace{\sum_{i \in J_0(u^*)\setminus I_0} x_i}_{=0} + \underbrace{\sum_{j \in I_1} (1-x_j)}_{=|I_1|} + \underbrace{\sum_{j \in J_1(u^*)\setminus I_1} (1-x_j)}_{=0} = |I| > 0. $$
    \qed
\end{proof}

\begin{remark}
    Note that $(ii)$ and $(iv)$ of Lemma \ref{lemma:1} provide valid inequalities at $u^*$.  
\end{remark}

Consider the concave minimization problem over $\K$ defined by:
\begin{equation}
\label{eq:Ptilde}
\min\{ p(x)~|~(x,y)\in \K \} \tag{$\widetilde{P}$}
\end{equation}
Problem \eqref{eq:Ptilde} is obviously DC and Proposition \ref{prop:1} is also true. Once a local minimizer of \eqref{eq:Ptilde} or \eqref{Pt} is verified (e.g., by Proposition \ref{prop:1}), then we have the following valid inequalities stated in Theorem \ref{thm:validineqbasedonq}, Lemma \ref{lemma:2} and Theorem \ref{thm:validineqbasedonPt}. 
\begin{theorem}
	\label{thm:validineqbasedonq}
	Let $u^* = (x^*,y^*)$ be a local minimizer of \eqref{eq:Ptilde}, then 
	\begin{equation}
	l_{u^*} (u) \geq l_{u^*}(u^*), \forall u\in \K.
	\end{equation}
\end{theorem}
\begin{proof}
	Since $u^*$ is a local minimizer of $p$ over $\mathcal{K}$, then there exists an open ball $B(u^*;r)$ centered at $u^*$ with radius $r>0$ such that
	\[p(u^*) \leq p(u), \; \forall u \in \mathcal{K} \cap B(u^*,r). \]	
	One gets from Lemma \ref{lemma:1} that  $$l_{u^*}(u^*) = p(u^*) \text{ and } l_{u^*}(u) \geq p(u), \forall u\in \K,$$ then 
	\begin{align*}
	l_{u^*}(u^*) \leq l_{u^*}(u), \; \forall u \in \mathcal{K} \cap B(u^*,r),
	\end{align*}
	It follows from the convexity of $l_{u^*}$ that $l_{u^*}(u^*) \leq l_{u^*}(u), \; \forall u \in \mathcal{K}$. \qed
\end{proof}

\begin{lemma} 
	\label{lemma:2}
	$\forall t\geq 0$, let $u^*(t)$ be a local minimizer of (\ref{Pt}), then
	\begin{equation}
	\label{condition1}
	c^{\top}x^*(t) + d^{\top}y^*(t) + tl_{u^*(t)}(u^*(t)) \leq c^{\top}x + d^{\top}y + tl_{u^*(t)}(u),~\forall u \in \mathcal{K}.
	\end{equation}
\end{lemma}
\begin{proof}
	Since $u^*(t) = (x^*(t),y^*(t))$ is a local minimizer of problem  (\ref{Pt}), then there exists an open ball $B(u^*(t); r)$ (with $r>0$) such that
	\begin{equation}\label{eq:lemma2_ineq01}
	  c^{\top}x^*(t) + d^{\top}y^*(t) + tp(u^*(t)) \leq c^{\top}x + d^{\top}y + tp(u), \; \forall u \in \mathcal{K} \cap B(u^*(t); r).  
	\end{equation}
	Based on Lemma \ref{lemma:1}, we have $l_{u^*(t)}(u^*(t)) = p(u^*(t))$ and $l_{u^*(t)}(u) \geq p(u), \forall u\in \K$. Then it follows from $t\geq 0$ and \eqref{eq:lemma2_ineq01} that   
	\begin{align*}
	c^{\top}x^*(t) + d^{\top}y^*(t) + tl_{u^*(t)}(u^*(t)) \leq c^{\top}x + d^{\top}y + tl_{u^*(t)}(u), \; \forall u \in \mathcal{K} \cap B(u^*(t); r).
	\end{align*}
	So that $u^*(t)$ is a local minimizer of the linear mapping $u\mapsto c^{\top}x + d^{\top}y + tl_{u^*(t)}(u)$, thus it is also a global minimizer and 
	$$
	c^{\top}x^*(t) + d^{\top}y^*(t) + tl_{u^*(t)}(u^*(t)) \leq c^{\top}x + d^{\top}y + tl_{u^*(t)}(u),~\forall u \in \mathcal{K}.
	$$ 
	 \qed
\end{proof}
    
\begin{theorem} 
	\label{thm:validineqbasedonPt}
	There exists a finite number $t_1 \geq 0$ such that for any $t > t_1$ and for any $u^*(t)$ local minimizer of problem (\ref{Pt}) obtained by DCA, we have \begin{equation}
	    \label{eq:validineqbasedonPt}
	    l_{u^*(t)}(u) \geq l_{u^*(t)}(u^*(t)), \forall u\in \K.
	\end{equation}
	Let $V^+(w) = \{u \in V(\mathcal{K})~|~l_{w}(w) - l_{w}(u) >0 \},$
	$M = \max_{u\in V(\K)}\{c^{\top}x + d^{\top}y\} -  \min_{u\in V(\K)\setminus \S}\{c^{\top}x + d^{\top}y\},$
	and $\sigma = \min_{w\in V(\K)\setminus \S} \min_{v\in V^+(w)} \{ l_{w}(w)-l_{w}(v) \}.$ A large enough $t_1$ can be computed by \begin{equation}\label{eq:t1}
	     t_1 = \frac{M}{\sigma}.
	 \end{equation}
\end{theorem}
\begin{proof}
	$\rhd$ If $u^*(t)\in \S$, then $l_{u^*(t)}(u^*(t))=0$ and we get from Lemma \ref{lemma:1} that $$l_{u^*(t)}(u) \geq 0 = l_{u^*(t)}(u^*(t)), ~\forall u\in \K.$$
	$\rhd$ Otherwise, $u^*(t)\notin \S$ and $l_{u^*(t)}(u^*(t))>0$. The simplex algorithm used in Step 2 of DCA assumes that $u^*(t)\in V(\K)$. Then $u^*(t)\in V(\K) \setminus \S$. Let $$ V^+(w) = \{u \in V(\mathcal{K})~|~l_{w}(w) - l_{w}(u) >0 \}.$$ 
	$(i)$ If $V^+(u^*(t))= \emptyset$, then $l_{u^*(t)}(u^*(t))\leq l_{u^*(t)}(u), \forall u\in V(\K)$, which implies (by linearity of $l_{u^*(t)}$ and $\K$) that $$l_{u^*(t)}(u^*(t))\leq l_{u^*(t)}(u), \forall u\in \K.$$
	$(ii)$ Otherwise, $V^+(u^*(t))\neq \emptyset$. Let  
	\begin{equation}
	    \label{eq:M}
	    M = \max_{u\in V(\K)}\{c^{\top}x + d^{\top}y\} -  \min_{u\in V(\K)\setminus \S}\{c^{\top}x + d^{\top}y\};
	\end{equation}
	\begin{equation}
	    \label{eq:sigma}
	    \sigma = \min_{w\in V(\K)\setminus \S} \min_{v\in V^+(w)} \{ l_{w}(w)-l_{w}(v) \}.
	\end{equation}
	Clearly, $0\leq M<+\infty$ and $0<\sigma <+\infty$. It follows from Lemma \ref{lemma:2} that 
	\begin{align*}
	  0< \sigma &\leq \min_{v\in V^+(u^*(t))} \{ l_{u^*(t)}(u^*(t))-l_{u^*(t)}(v) \} \\
	  &\leq l_{u^*(t)}(u^*(t))-l_{u^*(t)}(u), \forall u\in V^+(u^*(t)) \\
	  &\overset{\eqref{condition1}}{\leq}  \frac{c^{\top}x + d^{\top}y -  c^{\top}x^*(t) - d^{\top}y^*(t)}{t},~ \forall u\in V^+(u^*(t))\\
	  &\leq \frac{\max_{u\in V(\K)}\{c^{\top}x + d^{\top}y\} -  \min_{u\in V(\K)\setminus \S}\{c^{\top}x + d^{\top}y\}}{t} \leq  \frac{M}{t},
	\end{align*}
	which means
	$$0< t \leq \frac{M}{\sigma}<+\infty.$$
	Note that $M$ and $\sigma$ do not depend on $u^*(t)$ and $t$, so that if we take $t>\frac{M}{\sigma}=t_1$, then the assumption $V^+(u^*(t))\neq \emptyset$ will never be held and we have $V^+(u^*(t))= \emptyset$ which implies (by the case $(i)$) that $$l_{u^*(t)}(u^*(t))\leq l_{u^*(t)}(u), \forall u\in \K.$$ \qed
\end{proof}
\begin{remark}
	In the proof of Theorem \ref{thm:validineqbasedonPt}, we can compute $t_1=\frac{M}{\sigma}$ which is not difficult when the set $V(\K)$ is known.  Otherwise, the computation of $\sigma$ requires solving a nonconvex optimization problem \eqref{eq:sigma} which is in general intractable in practice. Note that if $t$ is not large enough, then the inequality \eqref{eq:validineqbasedonPt} may not be valid. A counterexample is given as follows:
\end{remark}
\begin{example}
	\label{eg:1}
	Consider the binary linear program with three binary variables:
	\begin{equation}
	\label{Ex-A}
	\begin{split}
	\min &\quad f(x) \defas -2x_1 - x_2 -x_3 \\
	\text{s.t.}& \quad  x\in \S = \K\cap \{0,1\}^3
	\end{split}
	\tag{Ex-A}
	\end{equation}
	where $\K = \{x\in [0,1]^3 ~|~ -3x_1 +  x_2 - 3x_3 \geq -3; - 3 x_1 -3 x_2 + x_3 \geq -3; -2x_1 -3x_2 \geq -3; -2x_1 - 3x_3 \geq -3\}$. The optimal solution is known as $(0,1,1)$ and $V(\K)=\{(0,0,0),(1,0,0),(0,1,0),(0,0,1),(0,1,1),(0.6,0.6,0.6)\}$. One can thus compute $\alpha_0=-2.4$, $m=1.2$, $\min_{x\in \K} f(x) = -2$ and $t_0=\frac{1}{3}$ 
	according to the formulation \eqref{eq:t0} in exact penalty Theorem \ref{thm:exactpenalty}. Now, let us take $t=1>t_0$, then an equivalent DC formulation of \eqref{Ex-A} (based on Theorem \ref{thm:exactpenalty}) is given by 
	\begin{equation}
	\label{Ex-At}
	\min \{f(x)+ t
	\sum_{i=1}^{3}\min\{x_i,1-x_i\} ~|~ x\in \K\}.
	\tag{Ex-At}
	\end{equation}
	Applying DCA for \eqref{Ex-At} with $t=1$ from initial point $x^0 = (0.6,0.6,0.6)$ will terminate at $x^0$ immediately, so that $x^0$ is a critical point which is also a local minimizer of \eqref{Ex-At} based on Proposition \ref{prop:1}. We will get the inequality \eqref{eq:validineqbasedonPt} at $x^0$ as $$x_1 + x_2 + x_3 \leq 1.8,$$ which is not valid for $(0,1,1)\in \K$. Therefore, $t=1$ is not large enough to get a valid inequality \eqref{eq:validineqbasedonPt} for $\K$ at $x^0$.
	
	Note that if $t$ is large enough, DCA will not terminate at $x^0$ any more, e.g., we can compute with $V(\K)$ the values of $M=2.4$, $\sigma=0.2$ and $t_1=\frac{M}{\sigma}=12$. Let us choose $t=13> \max\{t_0,t_1\}$, then starting DCA from initial point $x^0$ to problem \eqref{Ex-At}, we will not stop at $x^0$, but generate a sequence $\{x^k\}$ converging to $(0,1,1)$. The limit point is obviously a local minimizer for \eqref{Ex-At} based on Proposition \ref{prop:1}, and we obtain from Theorem \ref{thm:validineqbasedonPt} a valid inequality 
	$x_1-x_2-x_3\geq -2$ for $\K$ at $(0,1,1)$.  \qed
\end{example}

\subsection{DC cut at feasible critical point}\label{sec:3.2}

If a critical point $u^*$ of \eqref{Pt} obtained by DCA is included in $\S$, then we call that $u^*$ is a \emph{feasible critical point}; otherwise, it is an \emph{infeasible critical point}. In this subsection, we will establish a cutting plane  at feasible critical point $u^*$ by preserving all better feasible solutions in $\S$. 
\begin{theorem} \label{thm:dccut-type-I}
	Let $u^*\in\mathcal{S}$ be a feasible critical point of (\ref{Pt}) returned by DCA Algorithm \ref{algo:DCA}, then the inequality 
	\begin{equation} \label{Feasi1}
	l_{u^*}(u) \geq 1
	\end{equation}
	cuts off $u^*$ by preserving all better feasible solutions than $u^*$ in $\S$ \footnote{A better feasible solution than $u^*$ in $\S$ is a feasible solution in $\S$ whose objective value is smaller than $f(u^*)$.}. This cut is called the \textit{type-I DC cut} (cf. \dccuttypeI) at $u^*$.
\end{theorem}

\begin{proof}
	Since $u^*\in \S$, based on Lemma \ref{lemma:1} $(ii)$ and $(iii)$, we have the valid inequality $$l_{u^*}(u)\geq 0=l_{u^*}(u^*),~ \forall u\in \K.$$ Thus, $u=u^*$ does not hold the inequality $l_{u^*}(u)\geq 1$. Now, we will show that the inequality \eqref{Feasi1} holds for all better solutions than $u^*$ in $\S$: \\
	$\rhd$ Let us denote $\mathcal{C}_1\defas\{(x,y)\in \K~|~x=x^* \}$. Clearly, $\mathcal{C}_1\subset \S$. It follows from Lemma \ref{lemma:1} that $\forall u\in \mathcal{C}_1,\;l_{u^*}(u) = l_{u^*}(x^*) = 0$, so that the inequality $l_{u^*}(u)\geq 1$ cuts off $\mathcal{C}_1$ from $\S$. \\
	$\rhd$ Next, we can prove that $\C_1$ does not contain any better feasible solution than $u^*$. As $x^*\in \{0,1\}^n$, so that $u^*$ is a local minimizer of problem \eqref{Pt} based on Proposition \ref{prop:1}. Therefore, $\exists \mathcal{V}(u^*)$ a neighborhood of $u^*$ such that $\forall u\in \mathcal{V}(u^*)\cap \K$, we have 
	$$f(u^*)+tp(x^*) \leq f(u) + tp(u).$$
	Therefore, it follows from $\C_1\subset \S \subset \K$ that $f(u^*)+tp(x^*) \leq f(u) + tp(u), \forall u\in \mathcal{V}(u^*)\cap \C_1$, which implies that $u^*$ is also a local minimizer of problem
	\begin{equation}
	\min \{f(u) + tp(x) ~|~u\in \mathcal{C}_1 \}.
	\label{$P_1$}
	\end{equation}
	Based on the fact that $p(u)=0$ over $\C_1$, we can reduce problem \eqref{$P_1$} as a linear program $\min_{u\in \C_1} f(u)$. Then $u^*$ is a local minimizer of the linear program implies that $u^*$ is a global minimizer, thus $$f(u) \geq f(u^*),\;\forall u\in \mathcal{C}_1.$$ We conclude that $\C_1$ contains no better feasible solution than $u^*$. \\
	$\rhd$ On the other hand, $\forall u \in \mathcal{S}\setminus \mathcal{C}_1$, the inequality $l_{u^*}(u) \geq 1$ holds since $l_{u^*}(u)> 0$ (Lemma \ref{lemma:1} $(iv)$) and $l_{u^*}(u)\in \Z$. \\
	Consequently, the inequality $l_{u^*}(u) \geq 1$ cuts off $u^*$ (as well as $\C_1$) from $\S$ and contains all better feasible solutions than $u^*$ in $\mathcal{S}$. \qed
\end{proof}

\begin{remark}
	\begin{enumerate}
		\item The type-I DC cut will cut off at least one feasible solution $u^*$ in $\S$, therefore, it is a kind of \emph{local cut}. Note that this cut is possible to cut off some other feasible solutions in $\S$ (e.g., $\C_1$). But it is guaranteed in Theorem \ref{thm:dccut-type-I} that all removed feasible solutions can not be better than $u^*$.   
		\item If problem (\ref{Pt}) is a pure binary linear program, then the type-I DC cut will cut off only $u^*$ from $\mathcal{S}$ since $\C_1$ is reduced to the singleton $\{u^*\}$.
	\end{enumerate}	 
\end{remark}

The next Example \ref{eg:2} illustrates the type-I DC cut constructed at feasible critical points.
\begin{example}
	\label{eg:2}
	Consider the 2-dimensional binary linear program
	\begin{equation}
	\label{b}
	\min \{ f(x) \defas -x_1 - x_2 ~|~ x\in\S = \{0,1\}^2\cap \K\}
	\tag{Ex-B}
	\end{equation}
	where $\K=\{x\in [0,1]^2 ~|~  -4x_1 + 12 x_2 \geq -1; - 12 x_1 -4 x_2 \geq -13\}$. The vertex set $V(\K) = \{(0,0),(0,1),(0.25,0),(0.75,1),(1,0.25)\}$. According to the exact penalty Theorem \ref{thm:exactpenalty}, one can easily compute from \eqref{eq:t0} that $\alpha_0 = -1.75$, $m=0.25$, $\min_{x\in \S} f(x)=-1$ and $t_0=3$. Let us take $t = 4>t_0$, we get the equivalent DC formulation (\ref{b2}) as
	\begin{equation}
	\label{b2}
	\min_{x\in \K}  f(x) + t\sum_{i=1}^{2}\min\{x_i,1-x_i\} 
	\tag{Ex-Bt}
	\end{equation}
	There are two feasible critical points $(0,0)$ and $(0,1)$ of problem \eqref{b2}. Fig \ref{fig:1} illustrates the $\dccuttypeI$ constructed at $(0,0)$ and $(0,1)$ respectively. 
	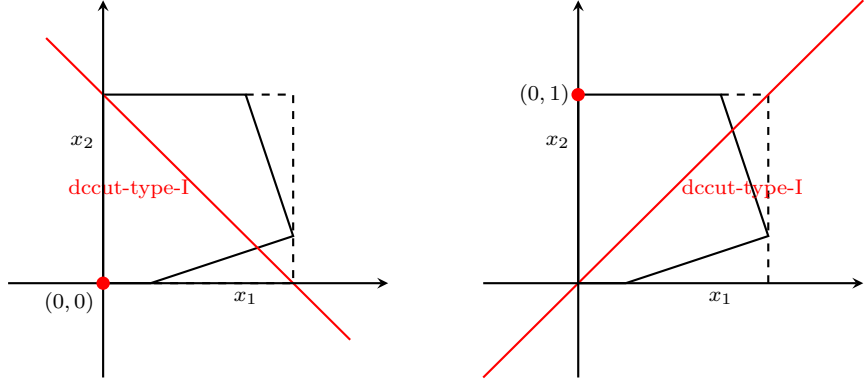
\begin{figure}[ht]
	\vskip -0.1in
		\caption{$\dccuttypeI$ at feasible critical points $(0,0)$ and $(0,1)$}
		\label{fig:1}
		\begin{tikzpicture}[thick,scale=2.5]
		\coordinate (A-03-03) at (-0.3,-0.3);
		\coordinate (A00) at (0, 0);
		\coordinate (A10) at (1,0);
		\coordinate (A01) at (0, 1);
		\coordinate (A11) at (1, 1);
		\coordinate (A0105) at (0,1.5);
		\coordinate (A1050) at (1.5,0);
		\coordinate (A0250) at (0.25, 0);
		\coordinate (A1025) at (1,0.25);
		\coordinate (A0751) at (0.75,1);
		\coordinate (A0-05) at (0,-0.5);
		\coordinate (A-050) at (-0.5,0);
		
		\coordinate (A-05105) at (-0.3,1.3);
		\coordinate (A105-05) at (1.3,-0.3);
		\coordinate (A-04103) at (-0.4,1.3);
		\coordinate (A10330) at (1.3333,0);
		\coordinate (A075125) at (0.75,1.25);
		
		\draw[solid] (A00) -- (A01) -- (A0751) -- (A1025) -- (A0250) -- (A00);
		\draw[dashed] (A0751) -- (A11) -- (A1025) -- (A10) -- (A0250);
		\draw[solid] (A0-05) -- (A00) -- (A-050);
		\draw[red] (A-05105) --node[anchor=east]{dccut-type-I} (A105-05);

		\draw [arrow](A00) -- node[anchor=east]{$x_2$} (A0105);
		\draw [arrow](A00) -- node[anchor=north]{$x_1$} (A1050);
		\draw (A00) node[anchor=north east]{$(0,0)$};
		\fill[red] (0,0) circle (1pt);
				
		\coordinate (B-050) at (2,0);
		\coordinate (B00) at (2.5, 0);
		\coordinate (B111) at (2.2,1);
		\coordinate (B0-025) at (2.5, -0.25);
		\coordinate (B0-05) at (2.5,-0.5);
		\coordinate (B10) at (3.5,0);
		\coordinate (B01) at (2.5, 1);
		\coordinate (B11) at (3.5, 1);
		\coordinate (B0105) at (2.5,1.5);
		\coordinate (B1050) at (4,0);
		\coordinate (B0250) at (2.75, 0);
		
		\coordinate (B1025) at (3.5,0.25);
		\coordinate (B0751) at (3.25,1);
		\coordinate (B0-05) at (2.5,-0.5);
		
		\coordinate (B-05105) at (2.2,1.3);
		\coordinate (B105-05) at (3.8,-0.3);
		\coordinate (B-04103) at (2.1,1.3);
		\coordinate (B10330) at (3.8333,0);
		
		\coordinate (B-1-1) at (2,-0.5);
		\coordinate (B105105) at (4,1.5);
		\coordinate (B1105) at (3.5,1.5);
		\coordinate (B-025-1) at (2.25,-1);
		
		\coordinate (B4025) at (3.5,0.25);
		\coordinate (B125025) at (3.75,0.25);
		
		\draw[solid] (B00) -- (B01) -- (B0751) -- (B1025) -- (B0250) -- (B00);
		\draw[dashed] (B0751) -- (B11) -- (B1025) ;
		\draw[dashed] (B10) -- (B4025);
		\draw[red] (B-1-1) --node[anchor=west]{dccut-type-I} (B105105);

		\draw[solid] (B-050) -- (B00) -- (B0-05); 
		\draw [arrow](B00) -- node[anchor=east]{$x_2$} (B0105);
		\draw [arrow](B00) -- node[anchor=north]{$x_1$} (B1050);
		\draw (B01) node[anchor=east]{$(0,1)$};
		\fill[red] (2.5,1) circle (1pt);
		\end{tikzpicture}
	\end{figure}
	\vskip -0.1in
\end{example}

\subsection{DC cut at infeasible critical point}\label{sec:3.3} 
Let $u^*\in V(\K)\setminus \S$ be an infeasible critical point of problem \eqref{Pt}. Based on the equivalence between \eqref{Pt} and \eqref{MBLP} for large enough $t$, if $u^*\in V(\K)\setminus \S$, then $u^*$ must not be a global minimizer of problem (\ref{Pt}), and even not assumed to be a local minimizer of \eqref{Pt} if some entries of $x^*$ equal to $\frac{1}{2}$. In this case, since the local optimality of $u^*$ can not be guaranteed, we can not apply the valid inequalities in Theorem \ref{thm:validineqbasedonq} or \ref{thm:validineqbasedonPt} to construct a DC cut at $u^*$. 

In this subsection, we will first discuss about the case when $u^*$ can be guaranteed as a local minimizer, then a DC cut based on Theorem \ref{thm:validineqbasedonq} or \ref{thm:validineqbasedonPt} can be easily constructed; Otherwise, we propose to introduce classical global cuts instead (such as Lift-and-Project cut and Gomory's cut). Therefore, we can always introduce cutting planes from $u^*$ by preserving better feasible solutions.    

Now, let us assume that $u^*\in V(\K)\setminus \S$ can be guaranteed as a local minimizer of \eqref{Pt} (e.g., based on Proposition \ref{prop:1}), then the next theorem provides a DC cut at $u^*$ for $\mathcal{S}$.
\begin{theorem} \label{thm:dccut-type-II}
	For any $t> t_1$ given by Theorem \ref{thm:validineqbasedonPt} and suppose that $u^*(t) \in V(\K)\setminus \mathcal{S}$ be a local minimizer of (\ref{Pt}) obtained by DCA satisfying $p(u^*(t))\notin \Z$, then the inequality   
	\begin{equation} \label{Inf_Noin}
	l_{u^*(t)}(u) \geq \lceil l_{u^*(t)}(u^*(t))\rceil \footnote{For $x\in \R$, the expression $\lceil x \rceil$ indicates the ceiling of $x$, i.e., the least integer greater than or equal to $x$.}
	\end{equation}
	cuts off $u^*(t)$ and holds for all points $u\in \mathcal{S}$. This cut is called the \emph{type-II DC cut} (cf. \dccuttypeII) at $u^*(t)$.
\end{theorem}

\begin{proof}
	Firstly, it follows from Lemma \ref{lemma:1} $(i)$ that $l_{u^*(t)}(u^*(t)) = p(u^*(t))$, then $p(u^*(t))\notin \Z$ implies $l_{u^*(t)}(u^*(t)) <  \lceil l_{u^*(t)}(u^*(t))\rceil$, thus $u^*(t)$ violates inequality (\ref{Inf_Noin}). Secondly, Theorem \ref{thm:validineqbasedonPt} implies that $l_{u^*(t)}(u) \geq l_{u^*(t)}(u^*(t)), \forall u\in \K$. Particularly, $\forall u \in \mathcal{S}\subset \K$, we have $\Z \ni l_{u^*(t)}(u)\geq l_{u^*(t)}(u^*(t))=p(u^*(t))\notin \Z$, then $l_{u^*(t)}(u) \geq \lceil l_{u^*(t)}(u^*(t))\rceil,~\forall u \in \mathcal{S}.$ \qed
\end{proof} 

\begin{remark}
	\label{rem:3}
	\begin{enumerate}
		\item The three assumptions required in Theorem \ref{thm:dccut-type-II}: (i) $t$ is large enough; (ii) $u^*(t)\in V(\K)\setminus \S$ is a local minimizer of \eqref{Pt}; (iii) $p(u^*(t))\notin \Z$, are non-negligible for constructing the type-II DC cut. Otherwise, if $t$ is not large enough and/or $u^*(t)$ is not a local minimizer of \eqref{Pt}, we observe in Example \ref{eg:1} that the inequality $l_{u^*(t)}(u) \geq l_{u^*(t)}(u^*(t))$ may not be valid for $\K$, and the $\dccuttypeII$ \eqref{Inf_Noin} may not be valid neither; If $p(u^*(t)) \in \Z$, then the inequality (\ref{Inf_Noin}) is not a cut at $u^*(t)$ since it contains $u^*(t)$. 
		\item Different to $\dccuttypeI$ which is a \emph{local cut}, the $\dccuttypeII$ is a \emph{global cut} which will not cut off any feasible point in $\S$. 
		\item We can also construct a $\dccuttypeII$ from a local minimizer $u^*$ of problem \eqref{eq:Ptilde} obtained by DCA satisfying $u^*\in V(\K)\setminus \S$ and $p(u^*) \notin \Z$. This DC cut is valid for $\mathcal{S}$ and cuts off $u^*$. 
	\end{enumerate}	
\end{remark}

Now, let $t>t_1$, for any $u^*\in V(\K)\setminus \S$ infeasible critical point obtained by DCA for \eqref{Pt} or \eqref{eq:Ptilde}, there are two following cases: 
\paragraph{\textbf{Case 1: If all entries of $x^*$ are not $\frac{1}{2}$ and $p(u^*) \notin \Z$.}}
Based on Proposition \ref{prop:1}, $u^*$ is a local minimizer of \eqref{Pt} or \eqref{eq:Ptilde}. Moreover with $t>t_1$ and $p(u^*)\notin\Z$, it follows from Theorem \ref{thm:dccut-type-II} that we have a type-II DC cut at $u^*$.

\paragraph{\textbf{Case 2: Otherwise (i.e., if any entry of $x^*$ is  $\frac{1}{2}$ or $p(u^*) \in \Z$).}} The assumptions of Theorem \ref{thm:dccut-type-II} are not verified, so that we can not construct a type-II DC cut at $u^*$ using Theorem \ref{thm:dccut-type-II}.

Constructing a DC cut for Case 2 is difficult and still be an open question in general. In some existing works \cite{Nguyen2006,Quang2010,Babacar2012}, a complicated procedure, namely \emph{Procedure P}, is proposed to search a DC cut for a special situation of Case 2. Procedure P performs as a branch-and-bound to search a point $z\neq u^*$ and $z\in\S$ verifying $l_{u^*}(z)=l_{u^*}(u^*).$ If $z$ does not exist, then the $\dccuttypeII$ can be constructed as the inequality \eqref{Inf_Noin} which will cut off $u^*$ without eliminating any feasible point (such as $z$) in $\S$. In all other cases, we still have no idea how to construct a DC cut properly. 

Unfortunately, the Procedure P performs, in worst cases, as a full branch-and-bound to solve a mixed-integer linear program just for constructing one DC cut, which is obviously unexpected due to its exponential complexity. Moreover, we may also fail to construct a DC cut when such a point $z$ does exist, and this bad case quite often occurs in many large-scale real-world applications. Therefore, the proposed DCA-CUT algorithm in those papers will be blocked in this case, and could be very inefficient even it works.

In next subsection, we will discuss about using classical global cuts, e.g., Lift-and-Project cut, to provide cutting planes in Case 2. 

\subsection{Lift-and-Project cut}
In this subsection, we are not intended to establish a new DC cut in Case 2 (which is quite sophisticated and almost impossible), but to propose some easily constructed classical global cuts for instead when they are applicable. Several possible choices are available, such as the Lift-and-Project (L\&P) cut, the Gomory's Mixed-Integer (GMI) cut and the Mixed-Integer Rounding (MIR) cut etc. Here, we will only briefly introduce the L\&P cut. The reader can refer to \cite{Balas1993,Marchand2002,Cornuejols2008} for more details about the L\&P cut and some other classical cuts. 

The idea of ``Lift-and-Project" is to consider problem \eqref{MBLP}, not in the original space, but in some space of higher dimension (lifting). Then valid inequalities
found in this higher dimensional space are projected back to the original space resulting in tighter mixed-integer programming formulations. A detailed procedure to produce an L\&P cut is described as follows.

\renewcommand{\algorithmcfname}{L\&P cut generation}
\begin{algorithm}[ht!]
	\caption{}
	\label{algo:LAP}
	\KwIn{$u^*\in V(\K)\setminus \mathcal{S}$.}
	\KwOut{L\&P cut.} 
	
	\textbf{Step 1: (Index selection)} Select an index $j \in \{1,\cdots,n\}$ such that $x_j^* \notin \Z$;
	
	\textbf{Step 2: (Cut generation)} 
	
	Set $C_j$ be an $m\times (n+q-1)$ matrix obtained from the matrix $[A | B]$ by removing the $j$-th column $a_j$;
	
	Set $\widetilde{D_j}$ be an $m\times (n+q)$ zero matrix with only $j$-th column being $a_j-b$;
	
	Set $\widetilde{C_j} \leftarrow [A|B] - \widetilde{D_j}$;	
	
	Solve the linear program:
	\begin{equation*}
	(w^*,v^*)\in\argmax\{ v^{\top}b - (w^{\top}\widetilde{D_j} + v^{\top}\widetilde{C_j})u^* ~|~w^{\top}C_j - v^{\top}C_j = 0, (w,v) \geq 0\}
	\end{equation*}
	An L\&P cut $({w^*}^{\top}\widetilde{D_j} + {v^*}^{\top}\widetilde{C_j})u \geq {v^*}^{\top}b$ is valid for $\mathcal{S}$ and cuts off $u^*$.
\end{algorithm}

	Note that an L\&P cut can be generated in theory at any $u^*\in V(\K)\setminus \S$ and with any index $j$ such that $x^*_j$ is fractional. Thus, one may generate several L\&P cuts at $u^*$ if there exists several fractional $x^*_j$. In practice, a deeper L\&P cut is preferred, which is often related to an index $j$ with maximum fractionality of $x^*_j$ (the closer to $0.5$ the better). Moreover, the minimum fractionality used to generate L\&P cuts is set to $0.001$ for numerical stability. Example \ref{eg:3} illustrates the differences between L\&P cut and DC cut.   
\begin{example}
	\label{eg:3}
	Consider the same problem in Example \ref{eg:2}, we have two infeasible critical points (also to be local minimizers) of problem (\ref{b2}) as $(0.75,1)$ and $(1,0.25)$ starting from which DCA will stop immediately at the starting point. The corresponding L\&P cuts and type-II DC cuts can be visualized in Fig \ref{fig:2}. 
	
	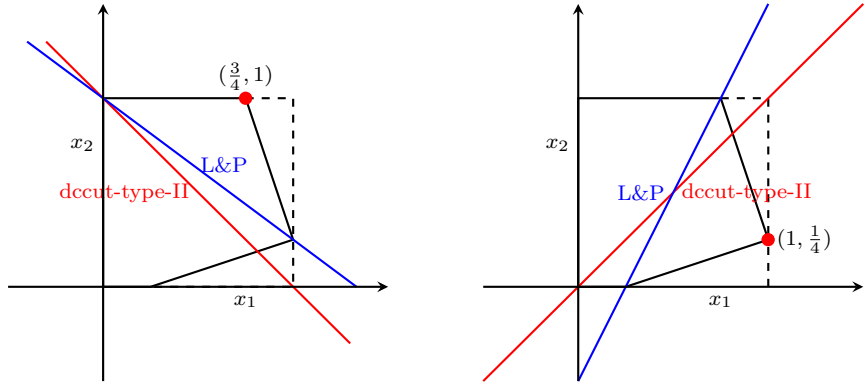
\begin{figure}[ht]
		\caption{$\dccuttypeII$ v.s. L\&P cut at infeasible local minima $(\frac{3}{4},1)$ and $(1,\frac{1}{4})$}
		\label{fig:2}
		\begin{tikzpicture}[thick,scale=2.5]
		\coordinate (A00) at (0, 0);
		\coordinate (A10) at (1,0);
		\coordinate (A01) at (0, 1);
		\coordinate (A11) at (1, 1);
		\coordinate (A0105) at (0,1.5);
		\coordinate (A1050) at (1.5,0);
		\coordinate (A0250) at (0.25, 0);
		\coordinate (A1025) at (1,0.25);
		\coordinate (A0751) at (0.75,1);
		\coordinate (A0-05) at (0,-0.5);
		\coordinate (A-050) at (-0.5,0);
		
		\coordinate (A-05105) at (-0.3,1.3);
		\coordinate (A105-05) at (1.3,-0.3);
		\coordinate (A-04103) at (-0.4,1.3);
		\coordinate (A10330) at (1.3333,0);
		\coordinate (A07510) at (0.75,1.0);
		
		\draw[solid] (A00) -- (A01) -- (A0751) -- (A1025) -- (A0250) -- (A00);
		\draw[dashed] (A0751) -- (A11) -- (A1025) -- (A10) -- (A0250);
		\draw[solid] (A0-05) -- (A00) -- (A-050);
		\draw[red] (A-05105) --node[anchor=east]{dccut-type-II} (A105-05);
		\draw[blue] (A-04103) -- node[anchor=west]{L\&P} (A10330);
		\fill[color=red] (0.75,1) circle (1pt);

		\draw [arrow](A00) -- node[anchor=east]{$x_2$} (A0105);
		\draw [arrow](A00) -- node[anchor=north]{$x_1$} (A1050);
		\draw (A07510) node[anchor=south]{$(\frac{3}{4},1)$};
		
		\coordinate (B-050) at (2,0);
		\coordinate (B00) at (2.5, 0);
		\coordinate (B0-025) at (2.5, -0.25);
		\coordinate (B0-05) at (2.5,-0.5);
		\coordinate (B10) at (3.5,0);
		\coordinate (B01) at (2.5, 1);
		\coordinate (B11) at (3.5, 1);
		\coordinate (B0105) at (2.5,1.5);
		\coordinate (B1050) at (4,0);
		\coordinate (B0250) at (2.75, 0);
		
		\coordinate (B1025) at (3.5,0.25);
		\coordinate (B0751) at (3.25,1);
		\coordinate (B0-05) at (2.5,-0.5);
		
		\coordinate (B-05105) at (2.2,1.3);
		\coordinate (B105-05) at (3.8,-0.3);
		\coordinate (B-04103) at (2.1,1.3);
		\coordinate (B10330) at (3.8333,0);
		
		\coordinate (B-1-1) at (2,-0.5);
		\coordinate (B105105) at (4,1.5);
		\coordinate (B1105) at (3.5,1.5);
		\coordinate (B-025-1) at (2.25,-1);
		
		\coordinate (B4025) at (3.5,0.25);
		\coordinate (B125025) at (3.75,0.25);
		
		\draw[solid] (B00) -- (B01) -- (B0751) -- (B1025) -- (B0250) -- (B00);
		\draw[dashed] (B0751) -- (B11) -- (B1025) ;
		\draw[dashed] (B10) -- (B4025);
		\draw[red] (B-1-1) --node[anchor=west]{dccut-type-II} (B105105);
		\draw[blue] (B0-05) -- node[anchor=east]{L\&P} (B1105);

		\draw[solid] (B-050) -- (B00) -- (B0-05); 
		\draw [arrow](B00) -- node[anchor=east]{$x_2$} (B0105);
		\draw [arrow](B00) -- node[anchor=north]{$x_1$} (B1050);
		\draw (3.5,0.25) node[anchor=west]{$(1,\frac{1}{4})$};
		\fill[red] (3.5,0.25) circle (1pt);
		\end{tikzpicture}
	\end{figure}
	
	Firstly, at the point $(0.75,1)$, we can construct one type-II DC cut: $x_1 + x_2 \leq 1$ and one L\&P cut: $3x_1+4x_2 \leq 4$ (since only one non-integer entries in $(0.75,1)$). In this case, we observe that the DC cut is deeper than the L\&P cut since more areas are eliminated. Secondly, at the point $(1,0.25)$, one DC cut is $-x_1+x_2\geq0$ and one L\&P cut is $4x_1 - 2x_2 \leq 1$. In this case, we can not say which cut is deeper. For reducing more quickly and deeply the set $\K$ and $\S$, we suggest to introduce both DC and L\&P cuts at $u^*$ if they are both applicable. Otherwise, in Case 2 where DC cut is not applicable, we will use L\&P cut only. \qed
\end{example}

\section{DCCUT Algorithm} \label{sec:DCCUTAlgo}
In this section, we will use the local/global DC cuts and the classical global cuts discussed in previous section to establish a cutting plane  algorithm for solving \eqref{MBLP}, namely \verb|DCCUT| Algorithm. We also develop a cut-generator toolbox on MATLAB, namely \verb|CUTGEN|, for generating classical global cuts for \eqref{MBLP}, including Lift-and-Project (L\&P) cut, Gomory's Mixed-Integer (GMI) cut  and Mixed-Integer Rounding (MIR) cut. 

\subsection{DCCUT Algorithm}\label{subsec:dccut}
DCCUT Algorithm consists of constructing in each iteration the constructible DC cuts (type-I and type-II) and/or classical global cuts (e.g., L\&P cut) to reduce progressively the sets $\K$ and $\S$, which results a sequence of the reduced sets $\{\K^n\}_{n\geq 0}$ and $\{\S^n\}_{n\geq 0}$. Once we find a feasible solution in $\S^k$, then we update the upper bound (\UB). Once the linear relaxation on $\K^k$ is solved, then we can update the lower bound (\LB). If the linear relaxation on $\K^k$ provides a feasible optimal solution in $\S^k$, or $\K^k=\emptyset$, or the gap between $\UB$ and $\LB$ is small enough, then we terminate DCCUT algorithm and return the best upper bound solution. Note that DCA is used to find good upper bound solutions and generate DC cuts. 
\renewcommand{\algorithmcfname}{DCCUT Algorithm}
\begin{algorithm}[ht]
	\caption{}
	\label{algo:DCCUT}
	\KwIn{Problem \eqref{MBLP}; penalty parameter $t$; absolute gap tolerance $\varepsilon>0$.}
	\KwOut{Optimal solution $u_{opt}$ and its optimal value $f_{opt}$.} 
	
	\textbf{Initialization:} $k \leftarrow 0$; $(P^0)\leftarrow$\eqref{MBLP}; $\K^0 \leftarrow \K$; $\S^0 \leftarrow \S$; $\UB \leftarrow +\infty$; $\LB\leftarrow -\infty$; $u_{opt}\leftarrow null$; $f_{opt}\leftarrow null$;
	
	\While {$\UB -\LB \geq \varepsilon$}
	{
	Solve \ref{R(Pk)} to obtain its optimal solution $u^0$;
	
	\If{\ref{R(Pk)} is infeasible}
	{
	    \Return;
	}
	\If{$f(u^0) > \LB$}
	{
	    Update lower bound: $\LB \leftarrow f(u^0)$;
	}
	
	\SetKwIF{If}{ElseIf}{Else}{if}{then}{else if}{else}{end}
	\eIf{$u^0\in \S^k$}
	{
    	\If{$u^0$ is a better feasible solution than $u_{opt}$}
		    {
		        Update upper bound: $\UB \leftarrow f(u^0)$; 
		        
		        Update best solution: $u_{opt} \leftarrow u^0$; $f_{opt} \leftarrow \UB$;
		    }
		    \Return;
	}
	{
			Use DCA to problem \eqref{Ptk} from initial point $u^0$ to get a critical point $u^*$;

			Initialize cut pool: $\V^k \leftarrow \emptyset$.
			
			Add classical global cuts at $u^0$ (e.g., L\&P cut) to $\V^k$;
			
			\SetKwIF{If}{ElseIf}{Else}{if}{then}{else if}{else}{end}
			\eIf {$u^*\in \S^k$}
			{
				Add a $\dccuttypeI$ at $u^*$ to $\V^k$;
				
				\If{$u^*$ is a better feasible solution than $u_{opt}$}
				{
					Update upper bound: $\UB \leftarrow f(u^*)$; 
					
					Update best solution: $u_{opt} \leftarrow u^*$; $f_{opt} \leftarrow \UB$;
				}
			}
			{
				\SetKwIF{If}{ElseIf}{Else}{if}{then}{else if}{else}{end}
				\uIf {$p(u^*)\notin \Z$  and $x_i^* \neq \frac{1}{2}\;\forall i\in \{1,\ldots,n\}$}
				{
    				Add a $\dccuttypeII$ at $u^*$ to $\V^k$;
				}
				\Else
				{
					Add classical global cuts at $u^*$ (e.g., L\&P cut) to $\V^k$;
				}
			}
			
			$\K^{k+1}\leftarrow \K^k \cap \V^k$; $\S^{k+1}\leftarrow \S^k \cap \V^k$; $k \leftarrow k+1$;		
		}
	}
\end{algorithm}

In the detailed DCCUT algorithm, we use some notations as follows: The sequence of sets $\{\V^n\}_{n\geq 0}$ stands for the collection of cuts (including DC cuts and classical global cuts), where the element $\V^k$ denotes the cut pool constructed at iteration $k$. The sequences of the reduced sets $\{\S^n\}_{n\geq 0}$ and $\{\K^n\}_{n\geq 0}$ are initialized with $\K^0 = \K$ and $\S^0 = \S$, and updated by $\K^{k+1}=\K^k \cap \V^k$ and $\S^{k+1}=\S^k \cap \V^k$. The mixed integer program defined on the reduced set $\S^k$ at iteration $k$ is given by:
\begin{equation}
\label{Pk}
\min\{ f(u)~|~u\in \S^k \} \tag{$P^k$},
\end{equation}
and its linear relaxation denoted by \ref{R(Pk)} is defined as:
\begin{equation}
\label{R(Pk)}
\min\{ f(u)~|~u\in \K^k \} \tag*{$R(P^k)$}
\end{equation}
The DC formulation of problem \eqref{Pk} is denoted by \eqref{Ptk} defined by:
\begin{equation*}\label{Ptk}
\min\{f(u)+tp(u) ~|~u \in \K^k\}. \tag{$P^k_t$}
\end{equation*} 

Note that a large enough penalty parameter $t$ depends on $\K^k$ and $\S^k$. Based on the expressions of $t_0$ in \eqref{eq:t0} and $t_1$ in \eqref{eq:t1}, it is hard to say whether $t$ will be increased or decreased when $\K^k$ and $\S^k$ are reduced. However, there always exist large enough finite $t_0$ and $t_1$ for any $\K^k$ and $\S^k$. Therefore, we will suppose that $t$ is given large enough for all $\K^k$ and $\S^k$. 
The next theorem shows that DC cuts will never be introduced redundantly. 

\begin{theorem} \label{thm:7}
	No DC cut introduced in DCCUT Algorithm is redundant.
\end{theorem}
\begin{proof}
	$\rhd$ We will show that a type-I DC cut can not be redundant with a type-II DC cuts. Suppose that we have first introduced a type-I DC cut at $u^*\in \S$ as $l_{u^*}(u)\geq 1$, then we introduced a type-II DC cut at $v^*\in V(\K^k)\setminus \S^k$ as $l_{v^*}(u)\geq \lceil l_{v^*}(v^*) \rceil$ such that the two cutting planes are redundant. Then, the redundancy implies that \begin{equation}
	    \label{eq:thm8eq1}
	    l_{u^*}(u) = l_{v^*}(u), \forall u \in \K \text{ and } \lceil l_{v^*}(v^*) \rceil = 1.
	\end{equation}
	As a cutting plane at $v^*$, we also have 
	\begin{equation}
	    \label{eq:thm8eq2}
	    l_{v^*}(v^*) < \lceil l_{v^*}(v^*) \rceil.\end{equation} 
	    It follows from \eqref{eq:thm8eq1} and \eqref{eq:thm8eq2} that 
	$$l_{u^*}(v^*) = l_{v^*}(v^*) < \lceil l_{v^*}(v^*) \rceil = 1,$$
	which implies that $v^*$ is already been cut off by the first cutting plane $l_{u^*}(u)\geq 1$ at $u^*$, and vice-versa. Therefore, a type-I DC cut and a type-II DC cut can never be redundant. \\
	$\rhd$ It can be proved in a similar way that two type-I (resp. type-II) DC cuts can never be introduced redundantly. \qed
\end{proof}

\begin{remark}
	It is possible to introduce one DC cut at $u^*$ in form of $l_{u^*}(u)\geq \alpha$ and then introduce a new DC cut at $v^*\neq u^*$ in form of $l_{v^*}(u)\geq \beta$ with same affine parts ($l_{u^*}=l_{v^*}$) but with different right-hand sides ($\alpha\neq \beta$). Therefore, when the number of cutting planes is too large, cleaning those cuts with the same affine part maybe helpful to reduce the size of the cut pool.  
\end{remark}

\begin{theorem}[Convergence of DCCUT algorithm]
	DCCUT algorithm is convergent and terminates in one of the three cases:
	\begin{enumerate}
		\item[$\bullet$] decide that \eqref{MBLP} is infeasible;
		\item[$\bullet$] obtain an exact global optimal solution $u_{opt}$ of \eqref{MBLP};
		\item[$\bullet$] return an $\varepsilon$-optimal solution $u_{opt}\in\S$ of \eqref{MBLP} such that $\UB -\LB< \varepsilon$. 
	\end{enumerate}
\end{theorem}
\begin{proof}
	The convergence of DCCUT algorithm is guaranteed by the convergence of the classical cutting plane method with L\&P cuts (see e.g., \cite{Balas1993}), since DCCUT algorithm can be considered as a cutting plane method combining L\&P cuts with DCA, DC cuts and other global cuts. 
	
	The returns of DCCUT are given as follows: The first return located in the line $5$ of DCCUT algorithm corresponds to two different cases. The first case is $k=0$, i.e., the linear relaxation $R(P^0)$ is infeasible, thus problem \eqref{MBLP} is also infeasible; The second case is $k>0$, we will return the upper bound solution. In the later case, it is possible to find no solution, i.e., $\S=\emptyset$; Otherwise, the best upper bound solution $u_{opt}$ must be an exact global optimal solution of problem \eqref{MBLP}. The second return located in the line $15$ of DCCUT algorithm occurs if and only if the optimal solution of the linear relaxation \eqref{R(Pk)}, denoted by $u^0$, is a feasible solution in $\S^k$. In this case $u^0$ is a global optimal solution of problem \eqref{Pk}, then DCCUT is terminated, and the global optimal solution of problem \eqref{MBLP} is the better feasible solution between $u^0$ and the current upper bound solution $u_{opt}$.
	\qed
\end{proof}
\begin{remark}
	It is possible that our DCCUT algorithm, as the classical cutting plane algorithm for mixed-integer programs, converges in infinitely number of iterations. Note that there are several finite cutting plane algorithms for pure integer program (see e.g., \cite{gomory1960algorithm,Balas1993}), however, the finiteness is much harder to achieve in the mixed integer case. Gomory proposed the first finite cutting plane algorithm for mixed-integer programs with integer objective \cite{gomory1960algorithm}; Jeroslow \cite{jeroslow1980cutting} developed a finite cutting plane algorithm for mixed-integer programs in the context of facial disjunctive programming; Balas \cite{Balas1993} proved a finite cutting plane algorithm for mixed-binary programs using L\&P cuts; and J\"{o}rg \cite{jorg2008k} gave a finite cutting plane algorithm for mixed-integer programs with bounded polyhedra. In general, cutting plane algorithms may not be finite without particular assumptions in mixed-integer programs. Developing a finite DCCUT algorithm for \eqref{MBLP} still needs more investigations.
\end{remark}

\subsection{Variants of DCCUT}\label{subsec:vdccut}
In DCCUT algorithm, when the critical point $u^*\in V(\K^k)$ of problem \eqref{Ptk} obtained by DCA is infeasible to $\S^k$, then we will introduce either a type-II DC cut in line 28 or some classical global cuts in line 30 to cut off $u^*$ from $\K^k$. A possible variant of DCCUT, namely DCCUT-V1 algorithm, is to introduce both classical global cuts at $u^*$ if $u^*\notin \S^k$ and a type-II DC cut if the condition in line 27 is verified. This variant consists of more cuts in each iteration which could be helpful to improve lower bounds more quickly. However, more cutting planes will also increase the size of linear constraints in $\K^k$ thus slow down the iterations thereafter. Moreover, it is possible to generate inefficient cuts in some hard scenarios. Therefore, there is a trade off between the number of global cuts introduced in each iteration and the global performance of the cutting plane algorithm. In practice, we suggest to increase the number of cuts introduced in each iteration to analyse its impact to the global performance of cutting plane algorithms. Note that more cuts introdcued in each iteration will of course not change the convergence of DCCUT algorithm, so the convergence of DCCUT-V1 is guaranteed. 

\subsection{Parallel DCCUT}\label{subsec:pdccut}
We can also take advantage of parallel computing to further improve the performance of DCCUT algorithm and its variant. To this end, the classical global cuts at $u^*$ can be introduced in parallel since they are constructed independently. The corresponding parallel versions are referred to as P-DCCUT and P-DCCUT-V1, which are hopefully to improve the lower bounds more quickly and accelerate the computation in each iteration. 

Another possible parallel strategy is to execute the block from the line 17 to 32 in some synchronized parallel processes. Each parallel process could generate independently a subset of cutting planes which will be merged together to get the final cut pool $\V^k$ for the $k$th iteration at the end of the line 32 by synchronizing the returns of all parallel processes. Then, the best upper bound solution $u_{opt}$ and the upper bound $\UB$ have to be updated. Concerning the choice of initial points for DCA, suppose that we have $s~(\geq 1)$ available parallel workers, then one can use the initial point $u^0$ for the worker 1, and choose random initial points in $[0,1]^n\times [0,\bar{y}]$ for workers $2$ to $s$. The parallel process can help to potentially provide more feasible solutions quickly via restarting DCA to update upper bound, and generate more cutting planes simultaneously for lower bound improvement. 

\section{Experimental Results}\label{sec:Experiments}
In this section, we will report some numerical results of our proposed DCCUT algorithms. We implement these methods using MATLAB, and the linear subproblems are solved by GUROBI \cite{Gurobi}. Our codes are shared on Github \url{https://github.com/niuyishuai/DCCUT}. The numerical tests are performed on a cluster at Shanghai Jiao Tong University with $30$ CPUs (Intel Xeon Gold 6148 CPU @ 2.40GHz) and 256GB of RAM.

In order to measure the quality of the computed solutions and the performance of the cutting planes, we use the next two measures: the gap (the smaller the better) and the closed gap (the bigger the better). The gap (cf. \gap) is used to measure the quality of the upper bound solution defined as
$$\gap = \frac{\UB- \LB}{\max(|\UB|,|\LB|)+1},$$
where $\UB$ is the upper bound (i.e., the objective value at the best feasible solution) and $\LB$ is the lower bound (i.e., the optimal value of the last LP relaxation). The term $\max(|\UB|,|\LB|)+1$ aims to well define $\gap$ in the case where $|\UB|=|\LB|=0$. The closed gap (cf. \clgap) is used to measure the quality of the lower bound (i.e., the quality of cuts), that is $$\clgap = \frac{\LB-f^0}{\texttt{fbest}-f^0},$$ where $f^0$ is the optimal value of the initial LP relaxation $R(P^0)$, $\LB$ is the optimal value of the current LP relaxation \ref{R(Pk)}, and \texttt{fbest} is the best known solution so far to the problem. Note that \texttt{fbest} is different to $\UB$. The first term is the best known solution so far which is given by the user, while $\UB$ is the best solution computed using the current algorithm.  

In all tests, the absolute gap tolerance $\varepsilon$ in DCCUT algorithms is fixed to $0.01$ (i.e., algorithms will be terminated if $\UB-\LB\leq 0.01$). The tolerances for DCA are set to be $\varepsilon_1=10^{-6}$ and $\varepsilon_2=10^{-3}$. As we have discussed previously, a large-enough parameter $t$ used in DC program \eqref{Pt} is difficult to be determined exactly, however, based on our tests, the influence of $t$ to the numerical solutions and cutting plane generations are not sensitive at all. Fix $t$ as some large values in the interval $[500,1000]$, e.g., $t=500$, is often suitable for our test samples.   

\subsection{First Example: \texttt{sample\_30\_0\_10}}\label{subsec:firstexample}
Firstly, we test a pure integer program, namely \texttt{sample\_30\_0\_10}, consists of $30$ binary variables, $0$ continuous variable and $10$ linear constraints. The lower bound of this problem is difficult to be improved using cutting planes. The problem data is given by:

\noindent\resizebox{\columnwidth}{!}{
\centering $c^{\top} = \left[\begin{array}{cccccccccccccccccccccccccccccc}
-8 & 6 & 8 & -17 & -1 & -5 & 12 & 3 & -13 & 3 & 8 & 5 & 9 & -8 & 20 & -1 & 0 & 10 & 16 & -17 & 16 & -13 & -3 & -16 & 14 & 5 & 6 & -10 & -14 & 8 \end{array}\right]$,
}

\noindent\resizebox{\columnwidth}{!}{
\centering $A= \left[\begin{array}{cccccccccccccccccccccccccccccc} 5 & 8 & -9 & 5 & 10 & 0 & -5 & 2 & 0 & 3 & 6 & 1 & 0 & -2 & -1 & -1 & -4 & 9 & 3 & -3 & -9 & 7 & 5 & 1 & 2 & 1 & -7 & -1 & 5 & -5\\ 7 & -4 & -5 & -9 & -5 & -1 & 9 & 6 & 5 & -5 & -6 & -6 & 3 & 8 & -4 & -3 & 2 & 7 & -2 & -8 & -8 & -4 & -4 & -6 & -3 & 5 & 10 & -1 & 6 & -2\\ 9 & 10 & -7 & 5 & 2 & -5 & -1 & 5 & -9 & 8 & -3 & 2 & 7 & 5 & 7 & 1 & -2 & 2 & 8 & 7 & 4 & -7 & 10 & 6 & -1 & 0 & -2 & 6 & -9 & 1\\ -7 & -8 & 10 & -1 & -10 & -7 & 3 & -1 & -8 & -8 & -9 & 0 & -6 & 9 & 10 & -7 & 7 & -8 & -9 & 6 & -9 & 8 & -8 & 10 & 7 & 8 & -7 & 5 & 9 & -10\\ 9 & -6 & 10 & 8 & -1 & -10 & 0 & -3 & 3 & 0 & 2 & 0 & -3 & 6 & 6 & 1 & 8 & -7 & 4 & 10 & 6 & -5 & 4 & -3 & 10 & -6 & -4 & 5 & -3 & 3\\ -8 & 7 & 3 & 7 & 3 & 5 & -4 & 8 & 3 & 9 & 3 & 9 & -1 & -1 & 5 & 9 & -9 & 6 & -9 & -8 & 6 & -4 & -3 & -7 & 2 & 9 & 5 & -7 & 4 & -7\\ 6 & 4 & 9 & 4 & 1 & 10 & 3 & -5 & 0 & 5 & 2 & 0 & 0 & -2 & 6 & 6 & 4 & -4 & -5 & 6 & 2 & -8 & 3 & -1 & -8 & 2 & -2 & 0 & 7 & -4\\ -4 & -4 & 3 & -6 & -6 & 1 & 6 & 1 & -5 & 3 & -1 & -6 & -7 & -3 & -1 & 6 & 8 & -4 & -8 & -3 & -2 & 1 & 6 & 2 & 10 & 3 & -3 & 10 & 0 & -7\\ -9 & -3 & 2 & -1 & 7 & -10 & -1 & -10 & -3 & -10 & 8 & 10 & 9 & 0 & 7 & -6 & -5 & 4 & -7 & -3 & 4 & 7 & -4 & 0 & 1 & -3 & 8 & -7 & -6 & 5\\ 2 & 7 & 3 & 1 & -10 & 5 & 0 & -1 & 0 & -3 & -8 & 1 & 2 & -2 & 7 & -9 & -2 & 1 & 3 & 5 & -10 & 8 & 3 & -8 & 7 & 4 & 9 & 6 & 5 & 1 \end{array}\right],
$
}
and $b$ is a vector of $10$. 
The optimal value is -83. The classical cutting plane algorithm using L\&P cuts performs very bad since $\LB$ is hardly improved. 

Now, we are going to use this example to compare the improvement of the lower bound measured by $\clgap$ within 30 seconds using $6$ cutting plane algorithms: LAPCUT (the classical cutting plane algorithm with L\&P cuts only), DCCUT and DCCUT-V1 (with DC cut and L\&P cut only), as well as the parallel versions P-LAPCUT, P-DCCUT and P-DCCUT-V1 (using up to $30$ CPUs). Numerical results of $\gap$, $\clgap$ and $\UB$ are summarized in Table \ref{tab:gap_clgap_30s} where the column $\nLAP$ is the maximum number of L\&P cuts generated at each fractional point. The algorithm provided best $\clgap$ for fixed $\nLAP$ is highlighted in boldface, and the best records for DCCUT type algorithms and LAPCUT type algorithms are underlined. The influence of generating multiple L\&P cuts (\nLAP) to the improvement of the lower bound (\clgap) in different cutting plane algorithms are illustrated in Figure \ref{Fig:clgap}. 

\begin{table}[ht]
	\caption{Numerical results of different cutting plane algorithms (using up to 30 CPUs) for \texttt{sample\_30\_0\_10} within 30 seconds}
	\label{tab:gap_clgap_30s}
	\centering
		\resizebox{\columnwidth}{!}{
			\begin{tabular}{c||c|c||ccc|ccc|ccc|ccc} \hline
				\multirow{2}{*}{$\nLAP$} & LAPCUT & P-LAPCUT & \multicolumn{3}{c|}{DCCUT} & \multicolumn{3}{c|}{P-DCCUT} & \multicolumn{3}{c|}{DCCUT-V1} & \multicolumn{3}{c}{P-DCCUT-V1}\\
				\cline{2-15}
				& \clgap(\%) & \clgap(\%) &\gap(\%) & \clgap(\%) &$\UB$& \gap(\%) & \clgap(\%) &$\UB$& \gap(\%) & \clgap(\%) &$\UB$& \gap(\%) & \clgap(\%) &$\UB$\\
				\hline\hline
1 & 44.91 & 43.45 & 9.50 & 48.04 & -83.00 &9.77 & 46.39 & -83.00 &9.34 & \textbf{49.01} & -83.00 &10.24 & 43.48 & -83.00 \\
3 & 50.38 & 50.41 & 8.95 & 51.32 & -83.00 &8.95 & 51.31 & -83.00 &8.20 & \textbf{55.78} & -83.00 &8.27 & 55.34 & -83.00 \\
5 & 52.39 & 53.47 & 8.69 & 52.86 & -83.00 &8.15 & \textbf{56.07} & -83.00 &8.85 & 51.94 & -83.00 &8.64 & 53.18 & -83.00 \\
7 & 54.25 & 55.93 & 8.20 & 55.78 & -83.00 &7.53 & \textbf{59.69} & -83.00 &10.63 & 54.30 & -81.00 &7.74 & 58.46 & -83.00 \\
9 & 61.70 & 64.88 & 14.42 & 57.89 & -77.00 &6.99 & 62.80 & -83.00 &7.98 & 57.04 & -83.00 &6.62 & \textbf{64.90} & -83.00 \\
11 & 59.45 & 65.26 & 16.69 & 57.39 & -75.00 &6.52 & \textbf{65.45} & -83.00 &8.18 & 55.86 & -83.00 &6.57 & 65.15 & -83.00 \\
13 & 60.49 & \textbf{66.93} & 10.10 & 57.48 & -81.00 &6.56 & 65.22 & -83.00 &7.97 & 57.11 & -83.00 &6.86 & 63.51 & -83.00 \\
15 & 61.27 & \textbf{68.87} & 7.87 & 57.70 & -83.00 &6.30 & 66.73 & -83.00 &13.31 & 57.97 & -78.00 &6.39 & 66.22 & -83.00 \\
17 & 60.39 & \textbf{68.08} & 7.79 & 58.18 & -83.00 &6.08 & 67.93 & -83.00 &9.88 & 58.78 & -81.00 &6.32 & 66.62 & -83.00 \\
19 & 59.75 & \textbf{68.61} & 7.59 & 59.33 & -83.00 &6.10 & 67.82 & -83.00 &9.95 & 58.40 & -81.00 &6.31 & 66.63 & -83.00 \\
21 & 59.88 & 67.98 & 7.79 & 58.18 & -83.00 &6.04 & \textbf{68.19} & -83.00 &9.95 & 58.40 & -81.00 &8.58 & 66.44 & -81.00 \\
23 & 59.75 & 67.90 & 7.67 & 58.86 & -83.00 &5.99 & \textbf{68.44} & -83.00 &9.80 & 59.25 & -81.00 &8.61 & 66.25 & -81.00 \\
25 & 59.75 & \textbf{68.87} & 7.67 & 58.86 & -83.00 &5.93 & 68.80 & -83.00 &9.95 & 58.40 & -81.00 &6.39 & 66.17 & -83.00 \\
27 & 59.75 & \textbf{68.45} & 7.67 & 58.86 & -83.00 &6.07 & 67.97 & -83.00 &9.95 & 58.40 & -81.00 &8.58 & 66.42 & -81.00 \\
29 & 59.71 & \underline{68.88} & 7.85 & 57.80 & -83.00 &5.90 & \underline{\textbf{68.92}} & -83.00 &9.78 & 59.37 & -81.00 &8.45 & 67.17 & -81.00 \\
				\hline
		\end{tabular}}	
\end{table}

\begin{figure}[ht!] 
	\centering
\includegraphics[width=0.8\linewidth]{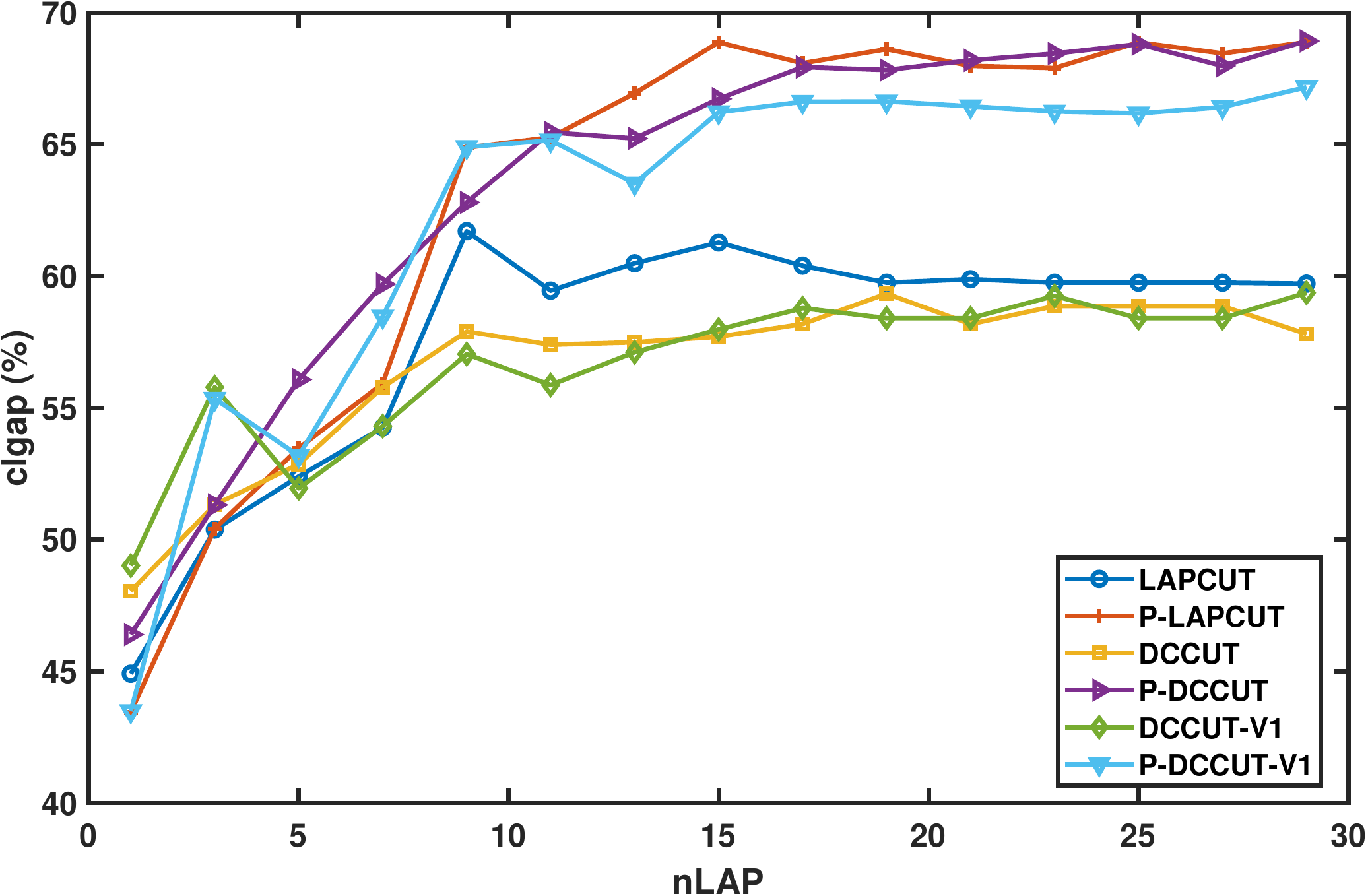}
	\caption{$\nLAP$ v.s. $\clgap$ for $\sampleone$ using 6 cutting plane algorithms in $30$ seconds}
	\label{Fig:clgap}
\end{figure}

\paragraph{Comments on numerical results in Figure \ref{Fig:clgap} and Table \ref{tab:gap_clgap_30s}}
\begin{itemize}
    \item The $\clgap$ is in general improved with the increase of $\nLAP$, and the parallel algorithms always perform better than the algorithms without parallelism. The algorithms with best (i.e., maximal) $\clgap$ seem to be P-LAPCUT and P-DCCUT, then followed by P-DCCUT-V1, LAPCUT, DCCUT-V1 and DCCUT. The closed gap of all compared methods is obviously improved at the beginning of the increase of $\nLAP$ (which is not a surprise since more than one cut added in each iteration could improve the lower bound more quickly), then the best $\clgap$ is reached at some $\nLAP$ and barely improved for larger $\nLAP$. The best choice for $\nLAP$ depends on the specific problems, the optimization methods and the computing resources. In general, if parallel computing resources are abundant enough, then the best $\nLAP$ could be found at the total number of integer variables (this test case is exactly an example). However, for large-scale integer optimization problems, setting $\nLAP$ to be the number of integer variables is generally impossible due to large number of variables, the limitation of resources and the cost of communications in parallel framework. In practice, for a given specific type of \eqref{MBLP}, it is always suggested to try more CPUs until the $\clgap$ is barely improved or all available CPUs are fully utilized. Note that, constructing LAP cuts too quickly is not always good for lower bound improvement, as more LAP cuts require solving more linear programs and increase the size of the linear constraints thereafter, thus could slow down the improvement in \clgap.   
    \item P-DCCUT can always find the global optimal solution -83 within 30 seconds, while the other DCCUT type algorithms can often find global optimal solution as well. This interesting result is due to the use of DCA which allows to quickly find a global optimal solution even the $\gap$ and $\clgap$ are far from optimal. However, this will no longer be the case for LAPCUT type algorithms, since a global optimal solution can only be found when the $\gap$ is reduced to $0$ or the $\clgap$ is increased to 1. This observation demonstrates that DCCUT type algorithms have the advantage of quickly finding and updating good upper bound solutions. 
\end{itemize}  

\subsection{Second Example: \sampletwo}
The second example is a pure binary linear program with $10$ binary variables and $10$ linear constraints. The data is given by:
{\scriptsize
$$c^{\top} = \left[\begin{array}{cccccccccc} -2.2696 & -0.3942 & 3.6018 & -0.9882 & -3.6844 & 2.9946 & -0.6465 & -0.4548 & 3.3601 & -2.1507 \end{array}\right],$$
$$A = \left[\begin{array}{cccccccccc} -0.9760 & -0.2597 & 0.3479 & -0.2245 & -1.3819 & -0.1699 & 1.8862 & 0.7490 & -1.0811 & -1.6861\\ -0.4179 & -0.5955 & -0.3371 & -0.3062 & 0.8015 & -1.4063 & 1.5717 & -1.2475 & -0.1122 & 1.4859\\ -0.7621 & 1.3052 & 0.5319 & 0.9497 & 1.6959 & 1.1078 & 0.6512 & 1.4769 & -0.1982 & -0.8773\\ -0.6498 & 0.0076 & -0.3066 & 1.8852 & 1.1138 & -0.7372 & 0.6152 & -0.5799 & -0.0836 & -0.0105\\ 0.0810 & -0.6236 & -0.8002 & 0.8480 & 1.6002 & -1.5630 & -0.5826 & 0.9058 & -1.2923 & -1.5472\\ -1.0805 & 1.5700 & 1.0696 & -0.0523 & -0.5775 & 1.9393 & 0.7844 & 1.3766 & -0.4787 & 0.6557\\ 1.9829 & -0.0374 & 1.3243 & 1.4588 & -1.8527 & -0.7158 & 0.0847 & 1.4510 & -0.5313 & 0.2882\\ -1.7562 & 1.8777 & 1.2762 & -1.1542 & 1.4269 & 1.9436 & -0.9917 & 1.8222 & -0.2950 & -1.4279\\ -0.5849 & 0.8489 & -0.9042 & 0.8495 & -0.6398 & -1.2533 & -0.5929 & 1.5026 & -1.0576 & -1.6837\\ -0.8693 & 1.5267 & 1.2275 & -0.6247 & 0.9537 & -0.8005 & -1.2797 & -0.3857 & 0.5152 & 0.0165 \end{array}\right],
$$
$$b^{\top} =  \left[\begin{array}{cccccccccc} 0.5288 & 0.7537 & 0.7413 & 0.1773 & 0.3005 & 0.2894 & 0.7848 & 0.2411 & 0.5476 & 0.2180 \end{array}\right],$$
}
and the optimal value is 0. 

We test $6$ cutting plane algorithms presented in subsection \ref{subsec:firstexample} for solving this problem. The absolute gap tolerance $\varepsilon$ is fixed to $0.01$ (i.e., algorithms will be terminated if $\UB-\LB\leq 0.01$). The updates of the $\UB$ (solid cycle line) and $\LB$ (dotted square line) with respect to the number of iterations are plotted in Figure \ref{fig:sampletwo}. The numerical results are summarized in Table \ref{tab:res_S2} where the best/worst computing time is highlighted in boldface/underline.
\begin{table}[h]
	\caption{Numerical results of different cutting plane algorithms (using up to 10 CPUs) for \texttt{sample\_10\_0\_10}}
	\label{tab:res_S2}
	\centering
		\resizebox{\columnwidth}{!}{
			\begin{tabular}{c|c|c|c|c|c|c|c|c|c|c} \hline
				Algorithms &\nLAP & \texttt{iter} &$\LB$&$\UB$& \gap(\%) & \clgap(\%) & cut\_dc1 & cut\_dc2 & cut\_lap & time(s)\\
				\hline\hline
DCCUT & 1 & 27 & 1.478e-02 & 0 & -1.46 & 100.22 & 1 & 13 & 40 & 1.156\\
P-DCCUT & 10 & 9 & 4.795e-01 & 0 & -32.41 & 107.17 & 1 & 4 & 86 & \textbf{0.833}\\
DCCUT-V1 & 1 & 29 & 2.060e-01 & 0 & -17.08 & 103.08 & 1 & 13 & 57 & 1.266 \\
P-DCCUT-V1 & 10 & 9 & 3.220e-01 & 0 & -24.36 & 104.82 & 1 & 4 & 99 & 1.127\\
LAPCUT & 1 & 65 & 0 & 0 & 0.00 & 100.00 & 0 & 0 & 65 & \underline{2.923}\\
P-LAPCUT & 10 & 16 & 0 & 0 & 0.00 & 100.00 & 0 & 0 & 83 & 1.356\\
				\hline
		\end{tabular}}	
\end{table}

\begin{figure}[ht!] 
	\subfigure[DCCUT]{
		\label{subfig:sampletwo_dccut} 
		\includegraphics[width=0.48\linewidth]{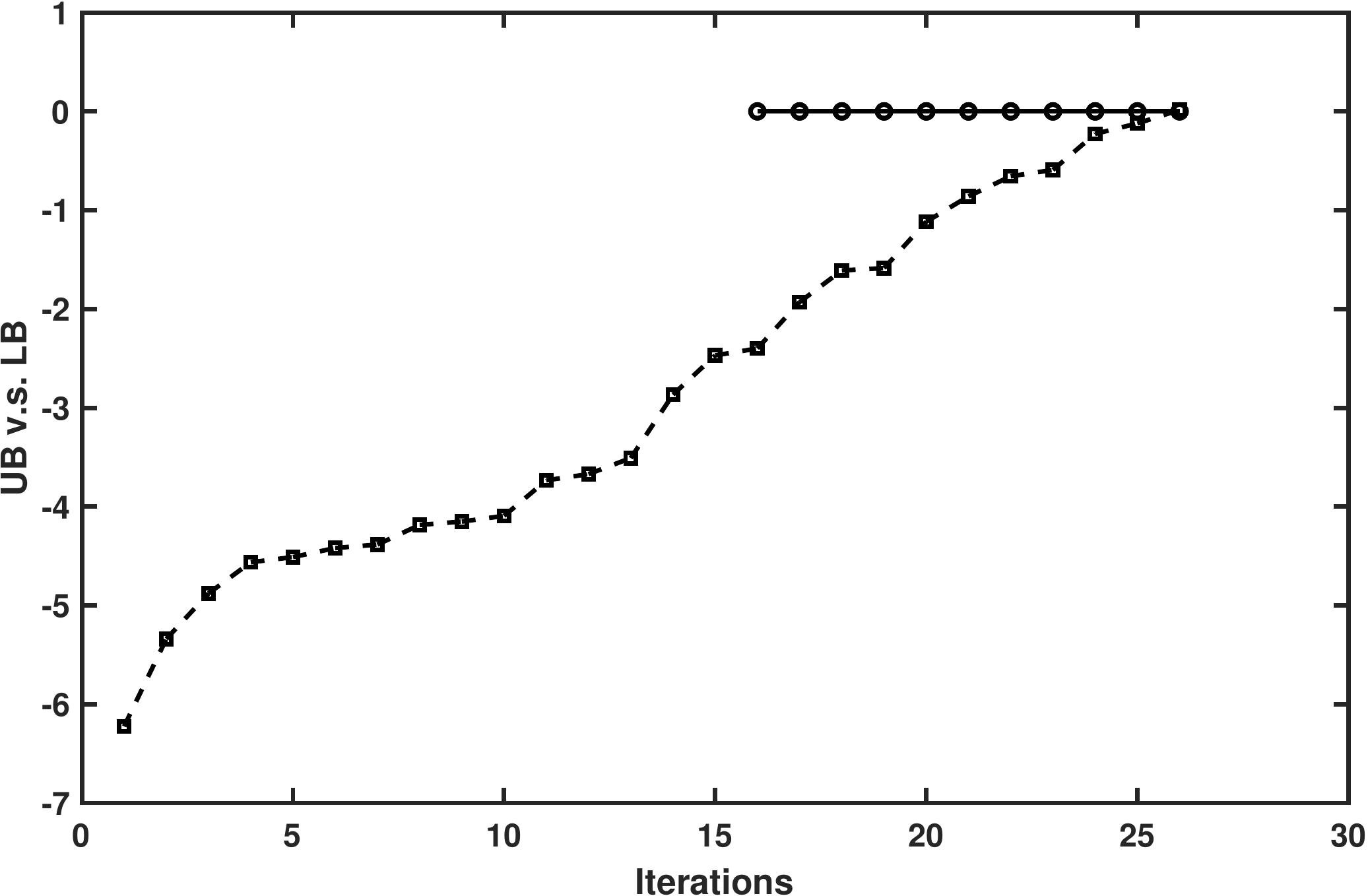}
	}
	\subfigure[P-DCCUT]{
		\label{subfig:sampletwo_pdccut} 
		\includegraphics[width=0.48\linewidth]{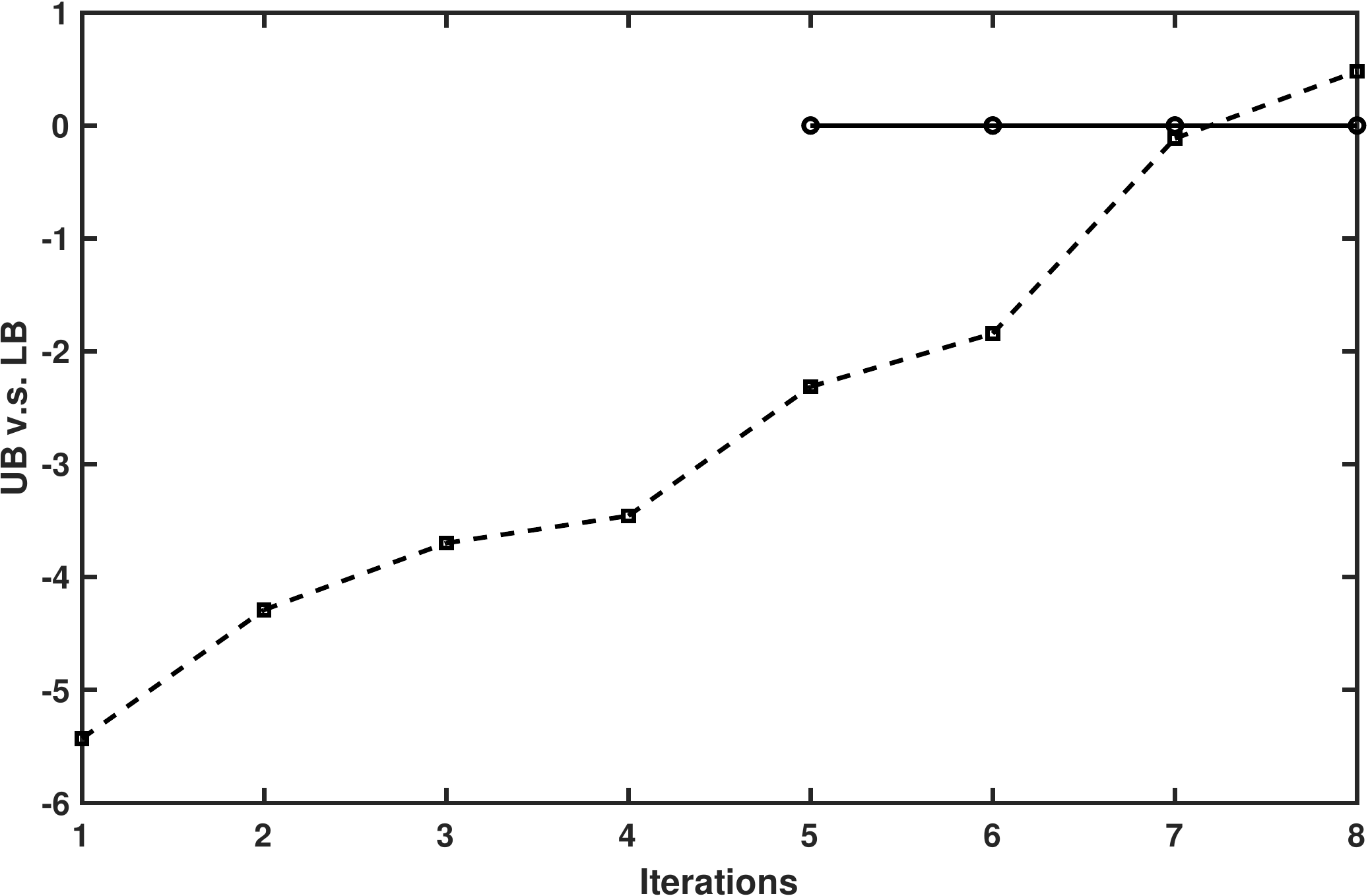}
	}
	\subfigure[DCCUT-V1]{
		\label{subfig:sampletwo_dccutv1} 
		\includegraphics[width=0.48\linewidth]{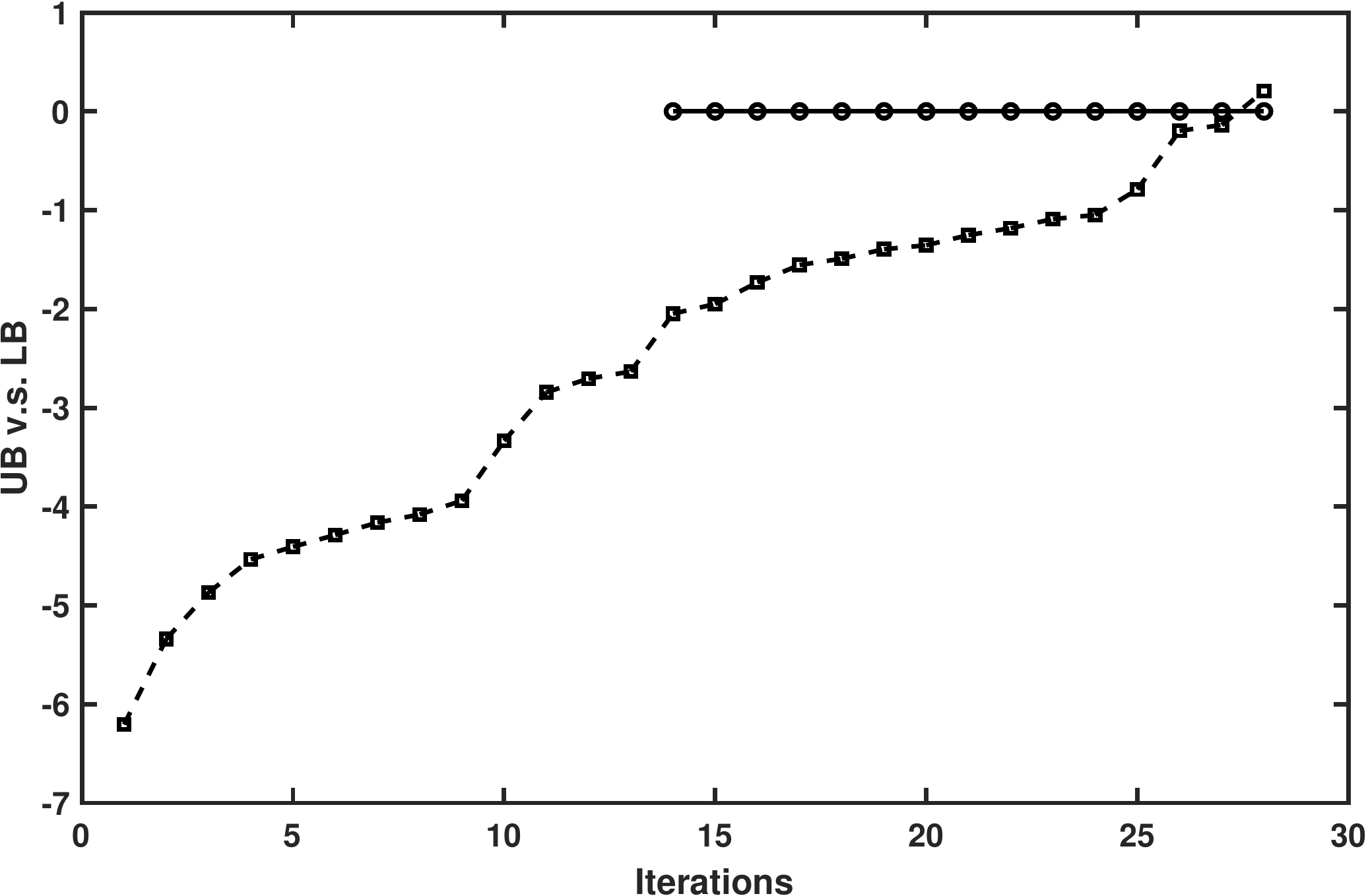}
	}
	\subfigure[P-DCCUT-V1]{
		\label{subfig:sampletwo_pdccutv1} 
		\includegraphics[width=0.48\linewidth]{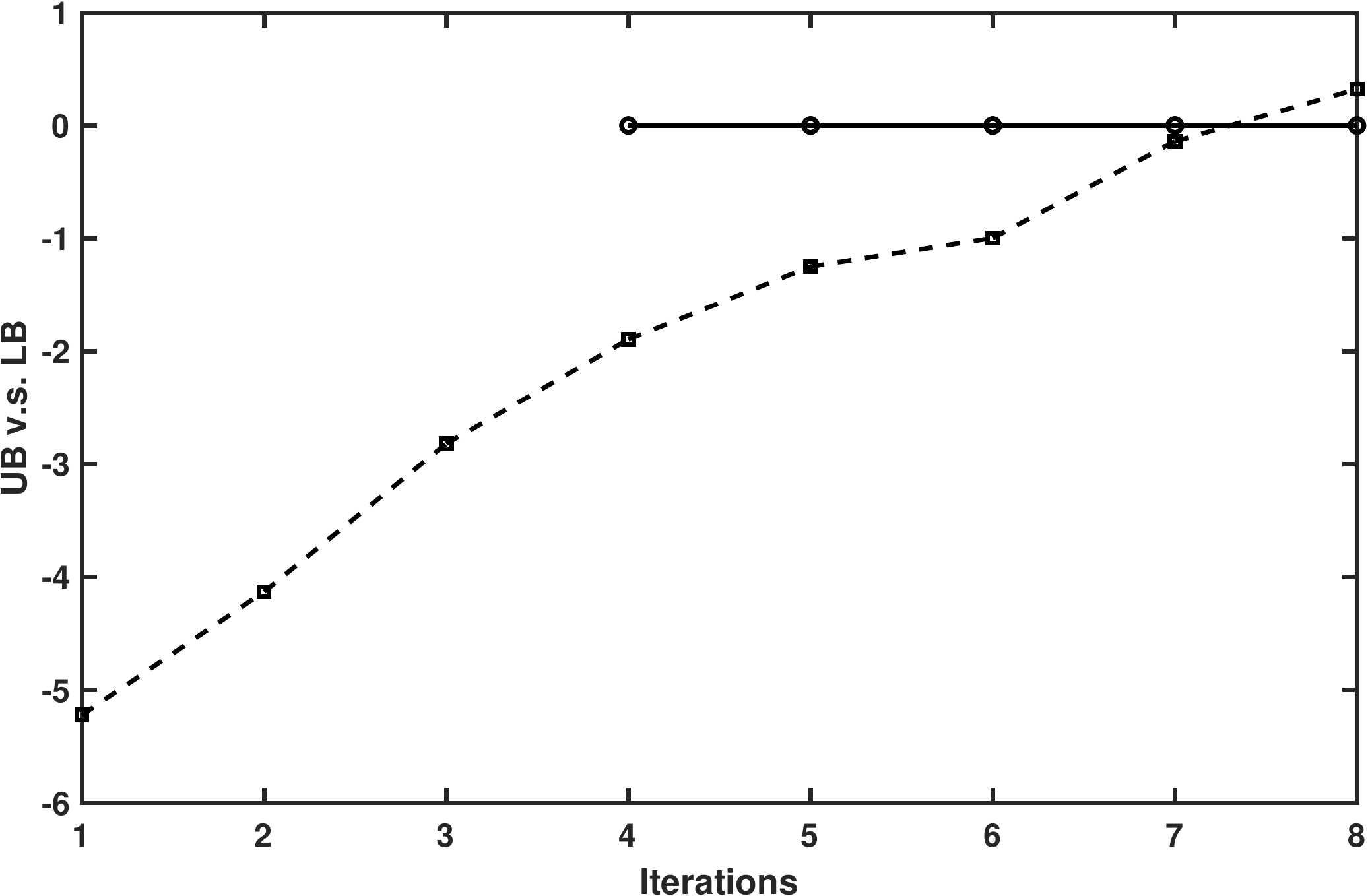}
	}
	\subfigure[LAPCUT]{
		\label{subfig:sampletwo_lapcut} 
		\includegraphics[width=0.48\linewidth]{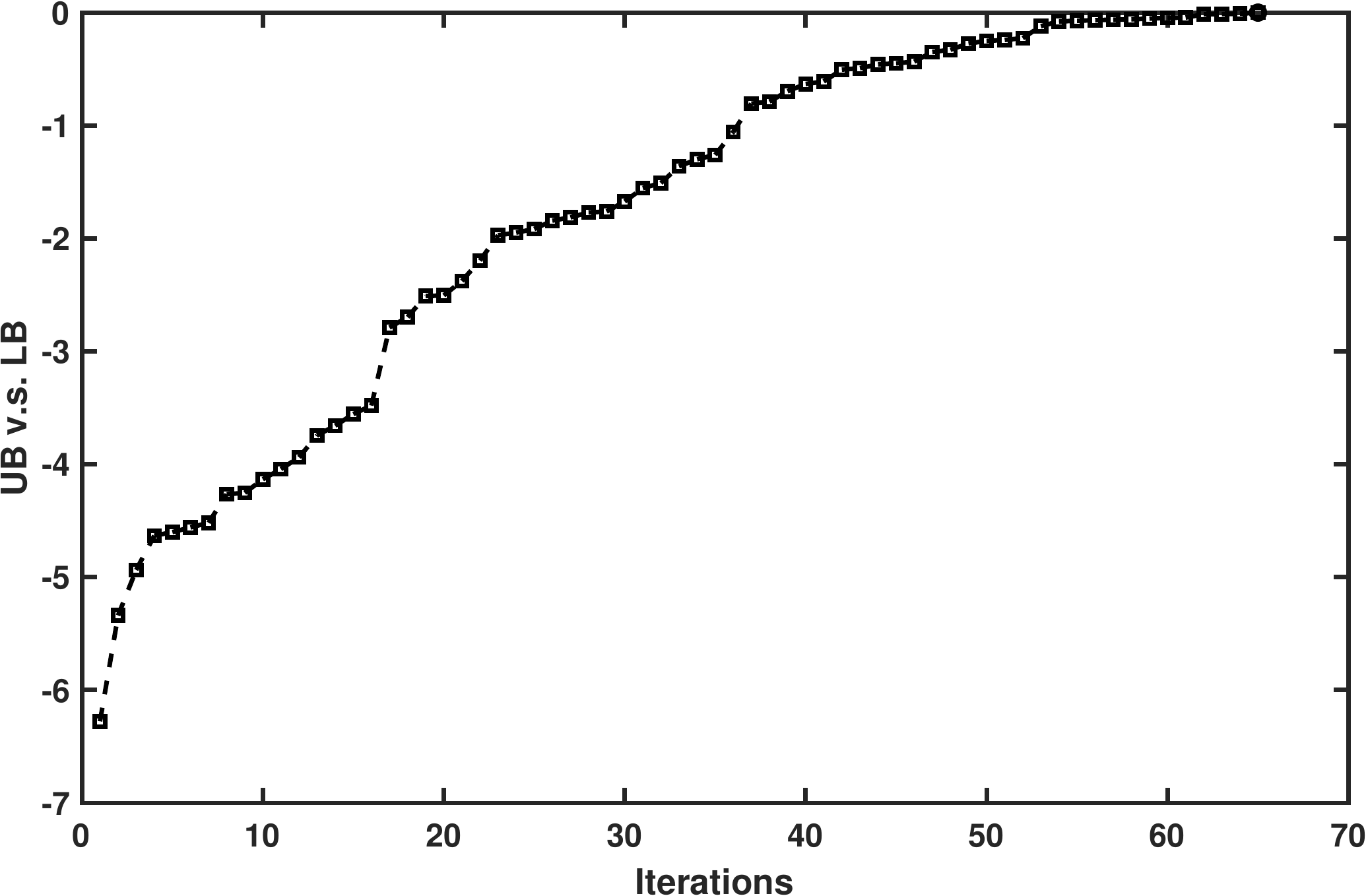}
	}
	\subfigure[P-LAPCUT]{
		\label{subfig:sampletwo_plapcut} 
		\includegraphics[width=0.48\linewidth]{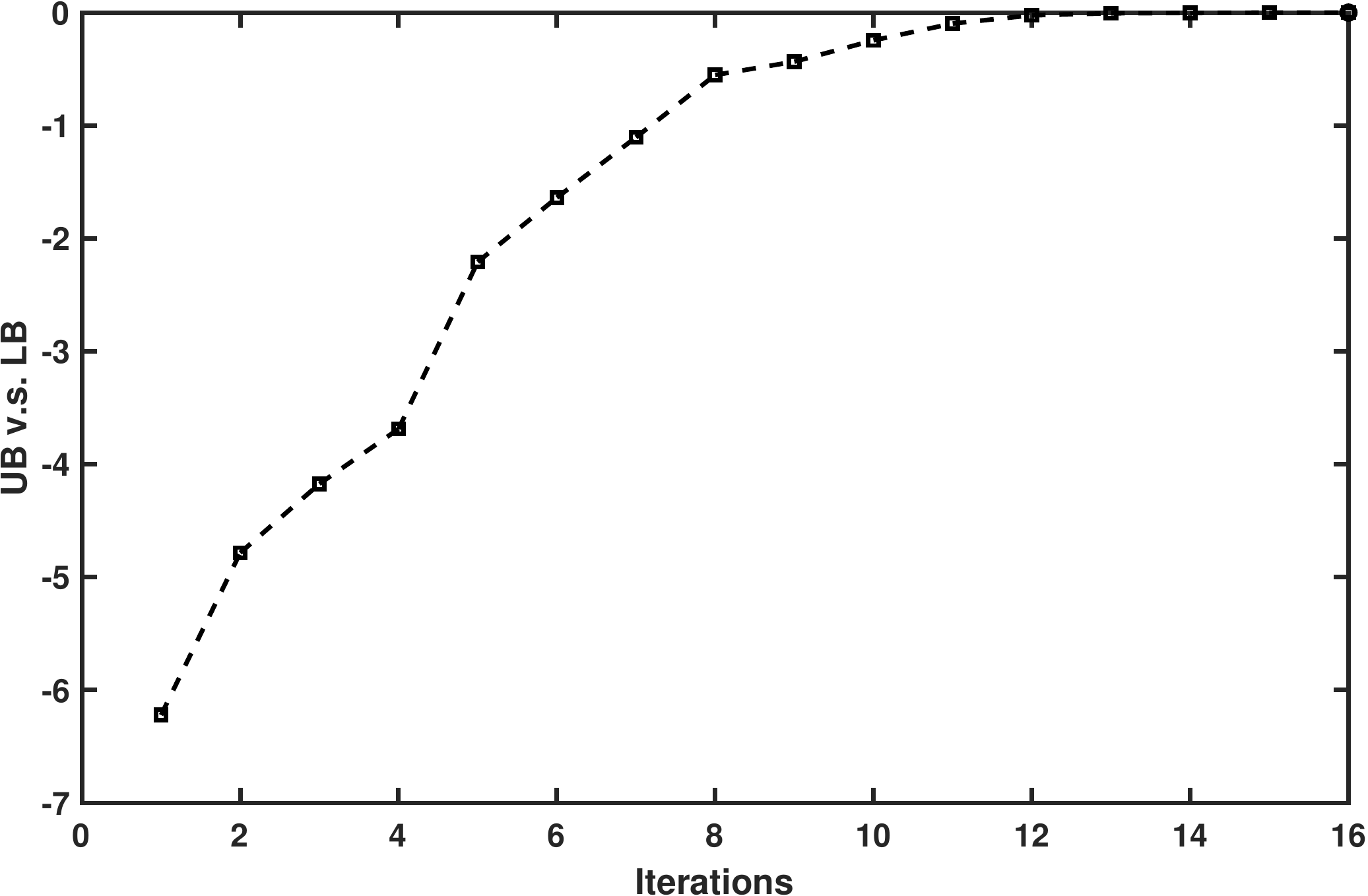}
	}
	\caption{Updates of $\UB$ and $\LB$ in different cutting plane algorithms for $\sampletwo$}
	\label{fig:sampletwo}
\end{figure}

\paragraph{Comments on numerical results in Figure \ref{fig:sampletwo} and Table \ref{tab:res_S2}} 
\begin{itemize}
    \item It is not surprise to see that the parallel algorithms always perform better than non-parallel ones. The fastest method is again P-DCCUT where $10$ CPUs are used for parallel workers. The slowest method is LAPCUT. 
    \item The $\gap$ for DCCUT type algorithms could be negative, similar to the $\clgap$ which could be greater than $100\%$. This is due to the fact that the global optimal solution is found by DCA when $\gap$ and $\clgap$ are not optimal yet, then type-I DC cut (local cut) is introduced to cut off the global optimal solution, thus the lower bound on the reduced set could be greater than the current upper bound, which leads to negative $\gap$ and more than $100\%$ $\clgap$. In both cases, the global optimality of the computed solution is guaranteed. Meanwhile, LAPCUT type algorithms can only have $\clgap$ between $0\%$ (initial) and $100\%$ (optimal), and the $\gap$ must be either $+\infty$ (non-optimal) or $0$ (optimal). 
    \item Among non-parallel algorithms (DCCUT, DCCUT-V1 and LAPCUT), LAPCUT creates most number of LAP cuts, while DCCUT and DCCUT-V1 require less LAP cuts. This observation demonstrates the benefit of DC cuts which indeed accelerates the convergence of cutting plane algorithm by improving the lower bound more quickly than using LAP cuts alone. 
    \item DCA is restarted in each iteration which not only helps to update upper bounds, but also leads to more DC cuts (type-I and type-II, cf. \texttt{cut\_dc1} and \texttt{cut\_dc2}) for improving the lower bounds. It is worth noting that, once a global optimal solution is found by DCA, then restarting DCA in subsequent iterations will not improve $\UB$ anymore, but it is still helpful to construct DC cuts for $\LB$ improvement. However, as no one know the global optimality of the current $\UB$, therefore, we cannot stop restarting DCA for $\UB$ updation until the convergence of the algorithm. In the case where $\UB$ is barely improved in many consecutive iterations, it is worth to think about reducing the probability of restarting DCA, which maybe useful to improve the overall performance of DCCUT type algorithms for solving hard cases. 
\end{itemize}  

\subsection{Performance on MIPLIB 2017 dataset}
Next, we will show some test results on the well-known benchmark testbed MIPLIB 2017 obtained from \url{http://miplib.zib.de}. We choose 14 problems in form of \eqref{MBLP} model whose information is given in Table \ref{tab:miplib_data}. Note that all equality constraints $A_{eq}x + B_{eq}y=b_{eq}$ are converted into inequalities as $A_{eq}x + B_{eq}y\leq b_{eq}$ and $-A_{eq}x - B_{eq}y\geq -b_{eq}$. We are interested in the numerical results of $\clgap$ and $\gap$ for 6 cutting plane algorithms (namely, LAPCUT, P-LAPCUT, DCCUT, P-DCCUT, DCCUT-V1 and P-DCCUT-V1) tested within $120$ seconds and with fixed $\nLAP=30$. The detailed numerical results are summarized in Table \ref{tab:miplib_gap_clgap_120s}, where the algorithm provided the best $\clgap$ is highlighted in boldface.
\begin{table}[ht]
	\caption{Tested problem information in MIPLIB 2017}
	\label{tab:miplib_data}
	\centering
			\begin{tabular}{c|c|c|c|c|c} \hline
				Problem & Binary & Continuous & Constraints & $f^0$ & \texttt{fbest}\\
			    \hline\hline
				neos5 & 53 & 10 & 63 & 13.00 & 15 \\
				mas74& 150 &1 &13 &10482.79528 & 11801.18572\\
 			    mad & 200 & 20 & 51 & 0 & 0.0268\\
				pk1 & 55 & 31 & 45 & 0& 11\\
				assign1-5-8& 130 &26 &161 &183.36 & 212\\
				ran14x18-disj-8 & 252 & 252 & 447 &3444.42 & 3712\\
				tr12-30& 360 & 720 &750 & 14210.43 &130596\\
				supportcase26& 396 & 40& 870 &1288.102161 &1745.123813\\ 
				exp-1-500-5-5& 250 &740 & 550 & 28427.05 &65887\\
				neos-3754480-nidda &50 &203 & 402 & -1216923.27&12941.74\\
				sp150x300d& 300 & 300 & 450 &4.89 &69\\
				mas76& 150 & 1 & 12 &38893.90364 &40005.05399\\
				neos-911970& 840 &  48 & 107 &23.26 &54.76\\
				gmu-35-40& 1200 & 5 & 424 &-2406943.556 &-2406733.369\\
				\hline
		\end{tabular}
\end{table}

\begin{table}[ht]
	\caption{Numerical results of different cutting plane algorithms (using up to 30 CPUs and $\nLAP=30$) for MIPLIB benchmark within 120 seconds}
	\label{tab:miplib_gap_clgap_120s}
	\centering
		\resizebox{\columnwidth}{!}{
			\begin{tabular}{c||c|c||ccc|ccc|ccc|ccc} \hline
				\multirow{2}{*}{Problem}  & LAPCUT & P-LAPCUT & \multicolumn{3}{c|}{DCCUT} & \multicolumn{3}{c|}{P-DCCUT} & \multicolumn{3}{c|}{DCCUT-V1} & \multicolumn{3}{c}{P-DCCUT-V1}\\
				\cline{2-15}
				&\clgap(\%) & \clgap(\%) &\gap(\%) & \clgap(\%) &$\UB$& \gap(\%) & \clgap(\%) &$\UB$& \gap(\%) & \clgap(\%) &$\UB$& \gap(\%) & \clgap(\%) &$\UB$\\
				\hline\hline
			    neos5 & 37.04 & \textbf{37.16} & 13.66 & 33.93 & 16.00 &7.88 & 36.94 & 15.00 &11.26 & 32.09 & 15.50 &7.91 & 36.73 & 15.00 \\
                \hline
                mas74 & 10.30 & \textbf{11.56} & Inf & 9.70 & Inf &Inf & 11.28 & Inf &Inf & 9.31 & Inf &Inf & 10.89 & Inf \\
                \hline
                mad & 0.00 & 0.00 & 44.35 & 0.00 & 0.80 &43.08 & 0.00 & 0.76 &44.35 & 0.00 & 0.80 &47.81 & 0.00 & 0.92 \\
                \hline
                pk1 & 0.00 & 0.00 & 96.30 & 0.00 & 26.00 &95.00 & 0.00 & 19.00 &96.30 & 0.00 & 26.00 &93.75 & 0.00 & 15.00 \\
                \hline
                assign1-5-8 & 18.79 & \textbf{21.12} & Inf & 16.81 & Inf &12.82 & 19.85 & 217.00 &Inf & 16.79 & Inf &Inf & 19.96 & Inf \\
                \hline
                ran14x18-disj-8 & 8.70 & 11.12 & Inf & 8.82 & Inf &Inf & 11.33 & Inf &Inf & 8.51 & Inf &Inf & \textbf{12.49} & Inf \\
                \hline
                tr12-30 & 46.17 & 56.30 & Inf & 53.45 & Inf &Inf & 65.84 & Inf &Inf & 72.47 & Inf &1.50 & \textbf{99.70} & 132228.00 \\
                \hline
                supportcase26 & 24.26 & 27.44 & Inf & 26.91 & Inf &Inf & \textbf{30.51} & Inf &Inf & 22.79 & Inf &Inf & 28.49 & Inf \\
                \hline
                exp-1-500-5-5 & 79.02 & 90.69 & Inf & 71.58 & Inf &Inf & 91.79 & Inf &Inf & 66.81 & Inf &0.28 & \textbf{99.51} & 65887.00 \\
                \hline
                neos-3754480-nidda & 49.42 & \textbf{57.80} & Inf & 44.97 & Inf &Inf & 56.77 & Inf &Inf & 40.64 & Inf &Inf & 52.46 & Inf \\
                \hline
                sp150x300d & 75.05 & \textbf{100.00} & Inf & 62.59 & Inf &Inf & 98.44 & Inf &Inf & 65.18 & Inf &0.00 & \textbf{100.00} & 69.00 \\
                \hline
                mas76 & 9.28 & \textbf{11.71} & Inf & 9.42 & Inf &Inf & 11.11 & Inf &Inf & 8.25 & Inf &Inf & 10.73 & Inf \\
                \hline
                neos-911970 & 82.74 & 91.68 & 62.64 & 83.40 & 134.24 &42.65 & 91.32 & 91.46 &62.64 & 83.40 & 134.24 &29.68 & \textbf{91.69} & 74.57 \\
                \hline
                gmu-35-40 & 0.00 & 0.02 & Inf & 0.03 & Inf &Inf & 0.05 & Inf &Inf & 0.00 & Inf &Inf & \textbf{0.06} & Inf \\
                \hline\hline
                \textbf{Avg clgap} & 31.48& 36.90& -- &30.12  & -- & -- & 37.52  & --  & &30.45 & -- & -- & \textbf{40.19}  & -- \\
                \hline
		\end{tabular}
		}	
\end{table}

\paragraph{Comments on numerical results in Tables \ref{tab:miplib_gap_clgap_120s}} 
\begin{itemize}
    \item It is not surprising to see that the parallel algorithms outperform their non-parallel versions by improving the updation of lower bounds (i.e., with larger $\clgap$). The best average $\clgap$ amounts to be P-DCCUT-V1, followed by P-DCCUT, P-LAPCUT, LAPCUT, DCCUT-V1 and DCCUT. 
    \item It is interesting to see that the parallel DCCUT-V1 (P-DCCUT-V1) achieves the best average $\clgap$, while the non-parallel version DCCUT-V1 gives the worst $\clgap$ in most of cases. This is probably due to the fact that P-DCCUT-V1 creates most cutting planes in each iteration, if these cuts can be processed in parallel, then the lower bound will be improved more quickly, which leads to the best $\clgap$ in P-DCCUT-V1; otherwise, it will take more time to create more cuts in each iteration for the non-parallel version, which leads to the worst performance in updating $\clgap$ within a limited time range. 
    \item P-DCCUT-V1 obtains more upper bound solutions than the other algorithms. We think this should be caused by the best improvement of the lower bounds, which reduces quickly the search region and thus increases the probability for DCA to find feasible local solutions.
    \item Moreover, in some cases, e.g. \texttt{tr12-30}, P-DCCUT-V1 and P-DCCUT outperform P-LAPCUT which demonstrates again the benefit of DC cut.
\end{itemize}

Based on the above tests, we also believe that the introduction of DC cut as a type of user cut in GUROBI and CPLEX solvers should be helpful in improving the lower bounds and the overall performance of these solvers. DC cuts and the corresponding algorithms should be promising techniques for mixed binary programs, and deserve more attention in the community.

\section{Conclusions and Perspectives}\label{sec:conclusion_perspective}

In this paper, we investigate the construction of two types of DC cuts (namely, $\dccuttypeI$ and $\dccuttypeII$). The type-I DC cut is a local cut for feasible point, while the type-II DC cut is a global cut for fractional point. We discuss about the cases where DC cuts are constructable. Otherwise, we propose introducing classical global cuts, such as Lift-and-Project cut, Gomory's mixed-integer cut and Mixed-integer rounding cut for instead. We give examples to illustrate the relationship between DC cuts with L\&P cuts. Combining DC cuts and classical global cuts, we establish a cutting plane algorithm (cf. DCCUT algorithm) for solving problem \eqref{MBLP}, whose convergence theorem is proved. A variant DCCUT algorithm (DCCUT-V1) by introducing more classical global cuts in each iteration, and parallel DCCUT algorithms are also proposed. Numerical results demonstrate that DCCUT type algorithms are able to find global optimal solution quickly without optimal $\gap$ or $\clgap$, and DC cut can indeed improve the lower bound. By introducing parallelism, the performance of DCCUT algorithms are significantly improved, and P-DCCUT-V1 algorithm often outperforms the others with best $\clgap$ tested on some MBLP problems in the MIPLIB2017 dataset.

Some questions deserve more attention: (i) developing finite mixed integer programming DCCUT algorithm; (ii) How to guarantee, without knowing the set of vertices $V(\K)$, that a given parameter $t$ is large enough for exact penalty and for type-II DC cut? (iii) proposing appropriate heuristic to restart DCA with a suitable probability for improving the overall performance of DCCUT type algorithms, especially to solve hard problems; (iv) introducing other classical cuts such as Gomory's mixed-integer cuts, mixed-integer rounding cut, knapsack cut, cover cuts and clique cuts in DCCUT algorithm to improve again the lower bound; (v) extending DC cut from binary linear case to general integer nonlinear case.

\begin{acknowledgements}
This work is supported by the Natural Science Foundation of China (Grant No: 11601327) and by the Key Construction National ``985" Program of China (Grant No: WF220426001).
\end{acknowledgements}

%
\section*{Conflict of interest}
The authors declare that they have no conflict of interest.

\bibliographystyle{spmpsci}      
\bibliography{reference}   

\end{document}